\documentclass[12pt]{amsart}
\usepackage{graphicx}
\usepackage[headings]{fullpage}
\usepackage{amssymb,epic,eepic,epsfig,amsbsy,amsmath,amscd,color}
\numberwithin{equation}{section}
                        \theoremstyle{plain}
\usepackage{mathrsfs}

\newcommand\no[1]{}

\newtheorem{theorem}{Theorem}[section]

\newtheorem{lemma}[theorem]{Lemma}
\newtheorem{corollary}[theorem]{Corollary}
\newtheorem{proposition}[theorem]{Proposition}

\newtheorem*{conjecture*}{Conjecture}
\newtheorem*{theorem*}{Theorem}
\newtheorem*{corollary*}{Corollary}

\theoremstyle{definition}
\newtheorem{remark}[theorem]{Remark}

\newtheorem{definition}[theorem]{Definition}

\def\BC{\mathbb C}

\def\BZ{\mathbb Z}

\def\BT{\mathbb T}
\def\BQ{\mathbb Q}

\def\CR{\mathcal R}

\def\la{\langle}
\def\ra{\rangle}

\DeclareMathOperator{\tr}{\mathrm tr}

\DeclareMathOperator{\SL}{\mathrm SL}

\def\ve{\varepsilon}
\def\be { \begin{equation} }
\def\ee { \end{equation} }

\begin{document}
\allowdisplaybreaks
\baselineskip16pt
\title{Adjoint Reidemeister torsions of once-punctured torus bundles}

\begin{abstract}
  Gang, Kim and Yoon have recently proposed a conjecture on a vanishing identity of adjoint Reidemeister torsions of hyperbolic 3-manifolds with torus boundary,
  from the viewpoint of wrapped M5-branes.
  In this paper, we provide infinitely many new supporting examples to this conjecture.
  These examples come from hyperbolic once-punctured torus bundles.
  We show that the vanishing identity holds for all hyperbolic once-punctured torus bundles with tunnel number one.
  We also show the vanishing identity does not hold for any torus knot exteriors.
\end{abstract}

\author{Anh T. Tran}
\address{Department of Mathematical Sciences,
The University of Texas at Dallas,
Richardson, TX 75080-3021, USA}
\email{att140830@utdallas.edu}

\author{Yoshikazu Yamaguchi}
\address{Faculty of Commerce,
  Waseda University,
  1-6-1 Nishiwaseda, Shinjuku-ku,
  Tokyo, 169-8050, Japan}
\email{shouji@waseda.jp}



\maketitle

\section{Introduction}
D.~Gang, S.~Kim and S.~Yoon proposed a vanishing identity of the adjoint Reidemeister torsion in~\cite{Gang}.
Their vanishing identity is based on the observation by \textit{3D--3D correspondence} between 3D supersymmetric quantum field theories and mathematics of 3-manifolds and knots.
 
\begin{conjecture*}[{\cite[Conjecture 1.1]{Gang}}]
  Let $M$ be a compact $3$-manifold with a torus boundary whose interior admits a hyperbolic structure.
  Suppose that the character variety of irreducible $\SL_2(\BC)$-representations of $\pi_1(M)$ consists of only
  irreducible components of dimension $1$. Then for any slope $\gamma \in H_1(\partial M; \BZ)$ we have
  \[\sum_{[\rho] \in \tr_{\gamma}^{-1}(z)} \frac{1}{\BT_{M, \gamma}(\rho)} = 0\]
  for generic $z \in \BC$. Here $\tr_\gamma$ is the trace function of $\gamma$ defined by $\tr_\gamma([\rho])=\tr \rho(\gamma)$ on the character variety
  and $\BT_{M, \gamma}(\rho)$ is the adjoint Reidemeister torsion with respect to $\rho$ and $\gamma$.
\end{conjecture*}

It was shown in~\cite{Gang} that the vanishing identity holds for the figure eight knot exterior.
Then Yoon has shown that the vanishing identity holds for all hyperbolic twist knot exteriors in~\cite{Yoon} and
furthermore extended the result of~\cite{Yoon} to all hyperbolic two--bridge knots in~\cite{Yoon2} with respect to a meridian slope.

The purpose of this paper is to study the vanishing identity of the adjoint Reidemeister torsions 
for hyperbolic 3-manifolds different from knot exteriors.
We investigate the vanishing identity for hyperbolic once-punctured torus bundles with tunnel number one
according to the observation in~\cite{Gang}.
Gang--Kim--Yoon's observation provides the description of the adjoint Reidemeister torsion as a rational function by using defining polynomials of character varieties
and proves the vanishing identity by the Jacobi's residue theorem.
We will show infinitely many new supporting examples to the conjecture by Gang, Kim and Yoon
in once-punctured torus bundles over the circle.

\begin{theorem*}
  The vanishing identity of adjoint Reidemeister torsion holds for all hyperbolic once-punctured torus bundles with tunnel number one.
\end{theorem*}

We also show that the inverse sum of adjoint Reidemeister torsions equals $\pm 2$ for all torus knot exteriors $M_{r, s}$ of type $(r, s)$ and all slopes $\gamma$.
\begin{theorem*}[Theorem~\ref{thm:sum_torusknot_general}]
  For any slope $\gamma = \mu^p \lambda^q$ and a generic $c \in \BC$, it holds that
  \[\sum_{[\rho] \in \tr_{\gamma}^{-1}(c)} \frac{1}{\BT_{M_{r, s}, \gamma}(\rho)} = \pm 2.\]
\end{theorem*}
\begin{corollary*}[Corollary~\ref{cor:nonvanish_torusknot}]
  The vanishing identity of adjoint Reidemeister torsion does not hold for any torus knot exteriors and any slopes.
\end{corollary*}

\section{Preliminaries}
\subsection{The fundamental groups of once-punctured torus bundles}
We almost follow the convention and notation used in~\cite{Baker-Petersen}.
Let $\beta$ and $\beta'$ denote curves on the once-punctured torus $T$ transversally intersecting once and
$\tau_a$ be the right-handed Dehn twist along the curve $a$.
We can regard the once-punctured torus bundles with tunnel number one, up to mirror images,
as a one-parameter family $\{M_n\}_{n \in \BZ}$ of the mapping tori of the homeomorphims
$\phi_n = \tau_{\beta'} \tau_\beta^{n+2}$ on $T$, that is,
\[M_n = T \times [0, 1] / (x, 0) \sim (\phi_n(x), 1).\]
If $n$ satisfies $|n| > 2$, then the once-punctured torus bundle $M_n$ is hyperbolic.

By abuse of notation, we use the same letters $\beta$ and $\beta'$ for the homotopy classes in $\pi_1(T, *)$ with the base point $* \in \partial T$.
According to~\cite{Baker-Petersen} the induced isomorphism $(\phi_n)_*$ maps $\beta$ and $\beta'$ 
to $\beta\beta'$ and $\beta'(\beta\beta')^{-n-2}$ respectively.
Then we have the following presentation of $\pi_1(M_n)$:
\begin{align*}
  \pi_1(M_n)
  &=  \la \beta', \beta, \mu \mid \beta' (\beta \beta')^{-(n+2)} = \mu \beta' \mu^{-1}, \,  \beta \beta' = \mu \beta \mu^{-1} \ra \\
  &= \la \alpha, \beta \mid \beta^{-n} = \alpha^{-1} \beta \alpha^2 \beta \alpha^{-1} \ra
\end{align*}
where $\mu$ stands for the homotopy class of an embedded curve in $\partial M_n$ transversally intersecting each fiber once and we put $\alpha = \beta^{-1}\mu$ in the second equality.

We call a curve representing $\mu = \beta \alpha$ {\it the meridian} of $M_n$.
Let $\lambda$ be the homotopy class represented by the boundary of a fiber corresponding $\partial T$.
Giving $\partial T$ the boundary orientation as $\partial T = \beta'\beta\beta'^{-1}\beta^{-1}$,
we have a pair of meridian-longitude such that
$\mu = \beta \alpha$ and $\lambda = \beta'\beta\beta'^{-1}\beta^{-1} = \alpha \beta \alpha^{-1} \beta \alpha \beta^{-1} \alpha^{-1} \beta^{-1}$,
which follows from $\beta'=\beta^{-1}\mu\beta\mu^{-1}=\alpha\beta\alpha^{-1}\beta^{-1}$.

\subsection{Review on the character varieties for once-punctured torus bundles}
Let $\rho$ denote a homomorphism from $\pi_1(M_n)$ into $\SL_2(\BC)$.
We call $\rho$ {\it an $\SL_2(\BC)$-representation} of $\pi_1(M_n)$.
We will observe $\SL_2(\BC)$-representations $\rho$ up to conjugation. 
This means that we will often replace an $\SL_2(\BC)$-representation $\rho$ by the composition with an inner automorphism of $\SL_2(\BC)$, that is $A \rho A^{-1}$ for some $A \in \SL_2(\BC)$. 

An $\SL_2(\BC)$-representation $\rho$ is referred to as being {\it irreducible}
if the standard action of $\rho(\pi_1(M_n))$ on $\BC^2$ has no nontrivial invariant subspace of $\BC^2$,
in other words, the $\SL_2(\BC)$-matrices $\rho(\alpha)$ and $\rho(\beta)$ are not conjugate to
upper triangular matrices simultaneously.
It is known that we can think of the set of conjugacy classes of irreducible $\SL_2(\BC)$-representations
as an algebraic set called {\it the character variety} (precisely, its Zariski closure is called the character variety for irreducible representations).
We will denote the character variety by $X(\pi_1(M_n))$.
The following functions on $X(\pi_1(M_n))$:
\[
x = \tr\rho(\alpha), \quad y = \tr\rho(\beta), \quad z = \tr\rho(\alpha\beta)
\]
give the structure of an algebraic set in $\BC^3$
when the fundamental group $\pi_1(M_n)$ are generated by two generators $\alpha$ and $\beta$
(for details see~\cite[Proposition~5.20]{Baker-Petersen}).
This means that the variables $x$, $y$ and $z$ play a role as a coordinate $(x, y, z)$ of the character variety $X(\pi_1(M_n))$.

We will review more details on the algebraic structure of the character varieties for once-punctured torus bundles with tunnel number one.
Let $f_n(y)$ be the Chebyshev polynomials defined by $f_0(y)=0$, $f_1(y)=1$ and $f_{n+1}(y) = y f_n(y) - f_{n-1}(y)$ for $n \in \BZ$.
Note that if $y=b+b^{-1}$ and $b \not= \pm 1$, then $f_n(y)$ can be expressed as 
\[
f_n(y) = \frac{b^n - b^{-n}}{b-b^{-1}}.
\]
According to~\cite[Proposition~5.23]{Baker-Petersen}
$\rho: \pi_1(M_n) \to \SL_2(\BC)$ is an irreducible representation if and only if
the above variables $x$, $y$ and $z$ satisfy 
\begin{align}
x^2 -1 + f_{n-1}(y) &=0, \label{eq:phi1_BP} \\
xz -y +f_{n}(y) &=0, \label{eq:phi2_BP}\\
x(f_{n+1}(y)-1) - z f_{n}(y) &=0. \label{eq:phi3_BP}
\end{align}
There are 3 cases on the parametrization of $X(\pi_1(M_n))$ (for details we refer to~\cite[Propositions~5.23 \&~5.32 and Remark~5.33]{Baker-Petersen}):
\begin{enumerate} 
\item \label{case:geometric}
  $x = \ve \sqrt{1-f_{n-1}(y)}, \, z = \ve \frac{y - f_n(y)}{\sqrt{1-f_{n-1}(y)}}$ where $\ve = \pm 1$ and $f_{n-1}(y) \not=1$.
\item $x=z=0$ and $y \in \CR_{n-2}$. Here $\CR_{n-2} =\{2\cos(2\pi k /(n-2)) \,|\,  k =0, \ldots, n-3\}$
  \label{case:fib}
  If we write $m^{\pm 1}$ and $\ell^{\pm 1}$ for the eigenvalues of $\rho(\mu)$ and $\rho(\lambda)$ respectively, then $m^2 = -1$ and $\ell =1$, so  $\tr \rho(\mu^p \lambda^q) \in \{\pm 2, 0\}$. 
\item  \label{case:extra}
  $x=y=0$ (this occurs only when $n \equiv 2 \pmod{4}$).
    If we write $m^{\pm 1}$ and $\ell^{\pm 1}$ for the eigenvalues of $\rho(\mu)$ and $\rho(\lambda)$ respectively, then it holds that $m^{4}\ell=1$. 
\end{enumerate}
The cases~\eqref{case:geometric} and \eqref{case:fib} form a parametrization of a hyperelliptic curve in $X(\pi_1(M_n))$.
We denote by $D$ this hyperelliptic curve.
The points in \eqref{case:fib} are given by $f_{n-1}(y) = 1$. We will see the details in Subsection~\ref{subsec:fac_chebyshev}.
Moreover the coordinate of a discrete faithful representation of $\pi_1(M_n)$ appears in the case~\eqref{case:geometric}.

It is also shown in~\cite[Theorem~5.1]{Baker-Petersen} that
the character variety $X(\pi_1(M_n))$ consists of the unique component $D$ when $n \not \equiv 2 \pmod{4}$ and
$X(\pi_1(M_n))$ consists of two components $D$ and $L$,
where $L$ is parametrized by the case~\eqref{case:extra} when $n \equiv 2 \pmod{4}$.
We observe the inverse sum of adjoint Reidemeister torsions on the components $D$ and $L$
for the conjecture by Gang--Kim--Yoon.
We will call the component $D$ the {\it geometric} component and $L$ the {\it extra} component.

\subsection{Factorization of Chebyshev polynomials}
\label{subsec:fac_chebyshev}
We review the details on the defining polynomials of $X(\pi_1(M_n))$ to investigate functions on it.
This section explains factorization of Chebyshev concerning the defining polynomials of $X(\pi_1(M_n))$.
\begin{definition}[Definition~4.8 in~\cite{Baker-Petersen}]
  Set
  \begin{align*}
    h_n(u) &= \left\{
    \begin{array}{rl}
      f_{m-1}(u) & \hbox{if $n=2m$} \\
      f_m(u) + f_{m-1}(u) & \hbox{if $n=2m+1$},
    \end{array}\right. \\
    j_n(u) &= \left\{
    \begin{array}{rl}
      f_m (u) & \hbox{if $n=2m$} \\
      f_{m+1} (u) + f_m (u) & \hbox{if $n=2m+1$},
    \end{array}\right. \\
    k_n(u) &= \left\{
    \begin{array}{rl}
      f_{m+2}(u) - f_m(u) & \hbox{if $n=2m$} \\
      f_{m+2}(u) - f_{m+1}(u) & \hbox{if $n=2m+1$},
    \end{array}\right. \\
    \ell_n(u) &= \left\{    
    \begin{array}{rl}
      f_{m+1}(u) - f_{m-1}(u) & \hbox{if $n=2m$} \\
      f_{m+1}(u) - f_m (u) & \hbox{if $n=2m+1$.}
    \end{array}\right. 
  \end{align*}  
\end{definition}
There are useful factorizations as follows.
\begin{lemma}[Lemma~4.9 in~\cite{Baker-Petersen}]
  For all $n$,
  \begin{align*}
    f_n(u) &= j_n(u)\ell_n(u), \\
    f_{n+1}(u) -1 &= j_n(u) k_n(u), \\
    f_{n-1}(u) -1 &= h_n(u) \ell_n(u), \\
    f_n(u) - u &= h_n(u) k_n(u),
  \end{align*}
  whence $(f_{n+1}(u)-1)(f_{n-1}(u)-1)=f_n(u)(f_n(u)-u)$.
\end{lemma}
We also review the common factors among $h_n(u)$ $j_n(u)$, $k_n(u)$ and $\ell_n(u)$.
\begin{lemma}[Lemma~4.13 in~\cite{Baker-Petersen}]
  \label{lemma:gcm_jlhk}
  The symbol $(a, b)$ denotes the ideal generated by polynomials $a(u)$ and $b(u)$ in $\BC[u]$.
  Then we have the following:
  \begin{enumerate}
  \item For all $n$, $(h_n, j_n)=(1)$.
  \item For all $n$, $\gcd(k_n, \ell_n)=(1)$.
  \item\label{item:corrected}
    If $n \not \equiv 2 \pmod{8}$, then $(h_n, k_n)=(1)$. Otherwise, $(h_n, k_n)=(u^2-2)$.
  \item If $n \equiv 0 \pmod{4}$, then $(j_n, k_n)=(u)$. Otherwise, $(j_n(u), k_n(u))=(1)$.
  \item If $n \equiv 2 \pmod{4}$, then $(h_n, \ell_n)=(u)$. Otherwise, $(h_n, \ell_n)=(1)$.
  \end{enumerate}
\end{lemma}
\begin{remark}
  We modified the assumption of~\eqref{item:corrected} from $n \not =2$ in~\cite[Lemma~4.13~$(3)$]{Baker-Petersen} to $n \not \equiv 2 \pmod{8}$.
\end{remark}
\begin{remark}
  When $n \equiv 2 \pmod{8}$, it also holds that $(u \pm \sqrt{2})^2 \nmid h_n(u)$ and $(u \pm \sqrt{2})^2 \nmid k_n(u)$.
\end{remark}

We can rewrite the defining equations~\eqref{eq:phi1_BP}, \eqref{eq:phi2_BP} and \eqref{eq:phi3_BP} of $X(\pi_1(M_n))$ as follows.
\begin{align}
x^2 + h_n(y)\ell_n(y) &=0, \label{eq:phi1_hell} \\
xz + h_n(y)k_n(y) &=0, \label{eq:phi2_hk}\\
j_n(y) (x k_n(y) - z \ell_n(y)) &=0. \label{eq:phi3_jkell}
\end{align}
Define $\hat \ell_n(y)$ and $\hat h_n(y)$ such that 
$\hat \ell_n(y) = \ell_n(y)$ and $\hat h_n(y) = h_n(y)$ when $n \not \equiv 2 \pmod{4}$, $y \hat \ell_n(y) = \ell_n(y)$ and $y \hat h_n(y) = h_n(y)$ when $n \equiv 2 \pmod{4}$.
The parametrization of $X(\pi_1(M_n))$ can be written as
\begin{enumerate} 
\item 
  $\displaystyle
  \big\{
  \big( \ve \sqrt{-h_n(y)\ell_n(y)}, y, \ve k_n(y) \sqrt{-\frac{\hat h_n(y)}{\hat \ell_n(y)}} \big)
  \,|\, \ve = \pm 1, \hat \ell (y) \not= 0, \hat h_n(y) \not= 0
  \big\} \subset D$.
  Note that $h_n(y) / \ell_n(y) = \hat h_n(y) / \hat \ell_n(y)$.
\item $\{(0, y, 0) \,|\, \hat h_n(y)=0\} \subset D$.
  The zeros of $\hat h_n(y)$ form $\CR_{n-2}$. 
\item 
  $\{ (0, 0, z)\,|\, z \in \BC \}$, which gives the line $L$ when $n \equiv 2 \pmod{4}$.
\end{enumerate}
Combining the above cases $(1)$ and $(2)$, we have the parametrization of the curve component $D$ in~\cite[Definition~5.31]{Baker-Petersen}, that is, $D$ is parametrized on $\BC^3$ by 
\[
\big\{
  \big( \ve \sqrt{-h_n(y)\ell_n(y)}, y, \ve k_n(y) \sqrt{-\frac{\hat h_n(y)}{\hat \ell_n(y)}} \big)
  \,|\, \ve = \pm 1, \hat \ell (y) \not= 0
  \big\}
\] 

We put a list of remarks needed in this paper.
\begin{remark}
  \label{rem:param_X}
  \begin{enumerate}
  \item The polynomials $\hat \ell_n(y)$ and $\hat h_n(y)$ have no common factors and no factor $y$, that is, $(\hat \ell_n, \hat h_n)=\BC[u]$, $\hat \ell_n(0) \not = 0$ and $\hat h_n(0) \not = 0$.
  \item It holds that $\hat \ell_n(y) \not= 0$ on $X(\pi_1(M_n))$.
    If $\hat \ell_n(y) = 0$, then we have $f_{n-1}(y)-1 = h_n(y) \ell_n(y) =0$ and $f_n(y)-y=k_n(y)h_n(y) \not= 0$ since neither $k_n(y)$ nor $h_n(y)$ has any common roots with $\hat \ell_n(y)$. The defining equations~\eqref{eq:phi1_BP} and~\eqref{eq:phi2_BP} do not hold simultaneously.  
  \item \label{item:intersectDL}
    In the case of $n \equiv 2 \pmod{4}$, the curve $D$ and the line $L$ in $X(\pi_1(M_n))$ intersect in the two points $(0, 0, \pm\sqrt{1/2 - 1/n})$ according to~\cite[Proposition~5.35]{Baker-Petersen}.
    Namely $X(\pi_1(M_n))$ is connected.
    \item If $n \not= 0$, then the degree of $f_n(u)$ is $|n|-1$.
  \end{enumerate}
\end{remark}

\subsection{General formula of the adjoint Reidemeister torsion for once-punctured torus bundles}
\label{subsec:general_formula_torsion}
We can find a general formula of the adjoint Reidemeister torsion for hyperbolic once-punctured torus bundles in~\cite[Section~4.5]{Porti}.
We review the general formula in Porti's book and then apply it to our situation.
For the details, we refer readers to~\cite[Section~4.5]{Porti}.

Let $M$ be a once-punctured torus bundle over the circle with monodromy $\phi$ and $T$ the fiber in $M$.
The character variety $X(\pi_1(M))$ is an algebraic subset in $\BC^3$ under the coordinate functions
\[x_1 = \tr\rho(g), \quad x_2 = \tr\rho(h), \quad x_3 = \tr\rho(gh)\]
where $g$ and $h$ are a pair of generators for $\pi_1(T)$.
This is derived from the fact that the character variety of the free group $\pi_1(T)=\la g, h \ra$ is isomorphic to
$\BC^3$ by corresponding the traces for $g$, $h$, $gh$ to
the coordinate $(x_1, x_2, x_3) \in \BC^3$.
The restriction map induces the homomorphism $i: \pi_1(T) \to \pi_1(M)$
and then the map $r: X(\pi_1(M_n)) \to X(\pi_1(T)) \simeq \BC^3$ is induced by the pull-back by $i$.
We can think of the image of $r$ as the fixed point set by the action of the monodromy $\phi$ on $X(\pi_1(T)) \simeq \BC^3$.
It has been shown in~\cite[Proposition~4.19]{Porti} that the adjoint Reidemeister torsion is determined by the eigenvalues
of the action induced by the monodromy $\phi$ on the tangent space of $X(\pi_1(T))$ at $r([\rho])$.
Here $[\rho]$ is the conjugacy class of an $\SL_2(\BC)$-representation $\rho$ of $\pi_1(M)$ and
the tangent space of $X(\pi_1(T))$ at $r([\rho])$ is the vector space $\BC^3$.
The action of the monodromy $\phi$ on $X(\pi_1(T)) \simeq \BC^3$ is given
by $2 \times 2$-matrix $A_\phi$ which represents the induced action of $\phi$ on $H_1(T;\BZ) \simeq \BZ^2$ as follows:
\begin{itemize}
\item the matrix $A_\phi$ is conjugate to $\pm R^{a_1} L^{a_2} \cdots R^{a_{k-1}}L^{a_k}$ ($a_1, \ldots, a_k \in \BZ$) where
  $L=\begin{bmatrix} 1 & 0 \\ 1 & 1 \end{bmatrix}$ and $R=\begin{bmatrix} 1 & 1 \\ 0 & 1 \end{bmatrix}$.
\item Let $P_i\, (i=1, 2, 3)$ be a polynomial in the variables $x_1$, $x_2$ and $x_3$ with coefficient $\BZ$.
  The actions corresponding to $L$ and $R$ on the coordinates $(P_1, P_2, P_3)$ are defined by
  \begin{align}
    (P_1, P_2, P_3)^L &= (P_3, P_2, P_2P_3-P_1), \label{eq:actionL}\\
    (P_1, P_2, P_3)^R &= (P_1, P_3, P_1P_3-P_2) \label{eq:actionR}
  \end{align}
  respectively.
\end{itemize}
We write $(P_1^W, P_2^W, P_3^W)$ for the action $(P_1, P_2, P_3)^W$ of a word $W$ in $L$ and $R$.
The polynomials $P_i^{A_\phi}$ $(i=1, 2, 3)$ are defined by the action $(P_1, P_2, P_3)^{A_\phi}$ with the initial condition $(P_1, P_2, P_3)^{\text{Id}} = (x_1, x_2, x_3)$. 
The action of the monodromy $\phi$ on the tangent space $T_{r([\rho])} X(\pi_1(T))$ is presented by
the Jacobian matrix $[\partial P_i^{A_\phi} / \partial x_j]_{i, j}$.
By~\cite[Lemma~4.21]{Porti}, the adjoint Reidemeister torsion $\BT_{M, \lambda}(\rho)$ is expressed as
\begin{equation}
  \label{eq:adT_general}
  \BT_{M , \lambda}(\rho) = 3 - \tr \left[ \frac{\partial P_i^{A_\phi}}{\partial x_j} \right]_{i, j}
\end{equation}
where $x_i \, (i=1, 2, 3)$ is evaluated at $r([\rho])$.

\section{Adjoint Reidemeister torsion}
We apply the general formula in Subsection~\ref{subsec:general_formula_torsion}
to hyperbolic once-punctured torus bundles with tunnel number one.
The monodromy $\phi_n = \tau_{\beta'} \tau_{\beta}^{n+2}$ induces the action given by $LR^{-(n+2)}$
on $H_1(T;\BZ)$.
Let $k$ be a non--negative integer.
By Eq.~\eqref{eq:actionR}, there are recurrence relations
\begin{align*}
(P_1^{LR^k} , P_2^{LR^k} , P_3^{LR^k} )
&= (P_1^{LR^{k-1}}, P_3^{LR^{k-1}}, P_1^{LR^{k-1}} P_3^{LR^{k-1}} - P_2^{LR^{k-1}})\\
\intertext{and}
(P_1^{LR^{-k}} , P_2^{LR^{-k}} , P_3^{LR^{-k}} ) 
&= (P_1^{LR^{-(k-1)}}, P_1^{LR^{-(k-1)}} P_2^{LR^{-(k-1)}} - P_3^{LR^{-(k-1)}}, P_2^{LR^{-(k-1)}}).
\end{align*}
The polynomials $P_1^{LR^{\pm k}}$, $P_2^{LR^{\pm k}}$ and $P_3^{LR^{\pm k}}$ satisfy
\begin{gather}
  P_1^{LR^k} = P_1^{L}, \quad
  P_2^{LR^k} = P_3^{LR^{k-1}}, \quad
  P_3^{LR^k} = P_1^{L} P_3^{LR^{k-1}} - P_3^{LR^{k-2}}  \label{eq:P_positive_k}\\
  \intertext{and}
  P_1^{LR^{-k}} = P_1^{L}, \quad
  P_2^{LR^{-k}} = P_1^{L} P_2^{LR^{-(k-1)}} - P_2^{LR^{-(k-2)}}, \quad
  P_3^{LR^{-k}} = P_2^{LR^{-(k-1)}}. \label{eq:P_negative_k}
\end{gather}
Both $P_3^{LR^k}$ and $P_2^{LR^{-k}}$ satisfy the same recurrence relation as the Chebyshev polynomial $f_k(P_1^L)$.
Note that any linear combination of Chebyshev polynomials also satisfies the same recurrence relation.
We regard the initial conditions of $P_3^{LR^k}$ and $P_2^{LR^{-k}}$ for $k=0, 1$ as
\begin{align*}
  P_3^L &= f_1(P_1^L)P_3^L, & P_3^{LR} &= P_1^{L} P_3^{L} - P_2^{L} = f_2(P_1^L) P_3^L - f_1(P_1^L)P_2^L \\
  \intertext{and}
  P_2^L &= f_1(P_1^L)P_2^L, & P_2^{LR^{-1}} &= P_1^{L} P_2^{L} - P_3^{L} = f_2(P_1^{L}) P_2^{L} - f_1(P_1^L)P_3^{L}
\end{align*}
since $f_2(y)=y$, $f_1(y)=1$, $f_0(y)=0$.
According to the initial conditions of $P_3^{LR^k}$ and $P_2^{LR^{-k}}$,
the polynomials $P_3^{LR^k}$ and $P_2^{LR^{-k}}$ in $P_1^L$ can be expressed as
\begin{align*}
  P_3^{LR^k} &= f_{k+1}(P_1^L)P_3^L - f_k(P_1^L) P_2^{L}, \\
  P_2^{LR^{-k}} &= f_{k+1}(P_1^L)P_2^L - f_k(P_1^L) P_3^{L}.
\end{align*}
By Eq.~\eqref{eq:actionL} the triple $(P_1^L, P_2^L, P_3^L)$ of polynomials is given by 
\[(x_1, x_2, x_3)^L = (x_3, x_2, x_2 x_3-x_1)\]
and then we get 
\begin{align}
  P_3^{LR^{k}}
  &=  (x_2 x_3 - x_1)  f_{k+1}(x_3)- x_2 f_k(x_3) = x_2 f_{k+2}(x_3)  - x_1  f_{k+1}(x_3), \label{eq:P3}\\
  P_2^{LR^{-k}}
  &= x_2 f_{k+1}(x_3) - (x_2 x_3 - x_1) f_k(x_3)  = -x_2 f_{k-1}(x_3) + x_1  f_k(x_3). \label{eq:P2}
\end{align}
Summarizing, the action of $LR^{\pm k}$ yields $(P_1^{LR^{\pm k}}, P_2^{LR^{\pm k}}, P_3^{LR^{\pm k}})$ as follows.
\begin{lemma}
  \label{lem:P_i}
  For a non-negative integer $k$, the polynomials $P_i^{LR^{\pm k}}$ $(i=1, 2, 3)$ are given by 
  \[
  (P_1^{LR^{\pm k}}, P_2^{LR^{\pm k}}, P_3^{LR^{\pm k}})
  = (x_3, x_2 f_{\pm k + 1}(x_3) - x_1 f_{\pm k}(x_3), x_2 f_{\pm k + 2}(x_3) - x_1 f_{\pm k +1}(x_3)).
  \]
\end{lemma}
\begin{proof}
  The polynomials $P_i^{LR^k}$ $(i=1, 2, 3)$ follow by Eqs.~\eqref{eq:P_positive_k} and~\eqref{eq:P3}.
  Since $f_{-n}(y)$ satisfies $f_{-n}(y) = - f_n(y)$ for any integer $n$, 
  it follows from Eqs.~\eqref{eq:P_negative_k} and~\eqref{eq:P2} that
  \begin{align*}
    P_1^{LR^{-k}} &= x_3, \\
    P_2^{LR^{-k}} &= -x_2 f_{k-1}(x_3) + x_1f_k(x_3) =  x_2 f_{-k+1}(x_3) - x_1f_{-k}(x_3),\\
    P_3^{LR^{-k}} &= P_2^{LR^{-(k-1)}} = x_2 f_{-k+2}(x_3) - x_1f_{-k+1}(x_3).
  \end{align*}
\end{proof}

Applying the general formula~\eqref{eq:adT_general} to $A_{\phi_n}=LR^{-(n+2)}$ together with Lemma~\ref{lem:P_i},
we obtain the adjoint Reidemeister torsion for the once-punctured torus bundle $M_n$ as follows.
\begin{lemma}
  \label{lem:adT}
  The adjoint Reidemeister torsion $\BT_{M_n, \lambda}(\rho)$ is expressed as
  \[
    \BT_{M_n, \lambda}(\rho)
    =3 + f_{n+1}(x_3) + x_2 \frac{d f_{n}(x_3)}{d x_3}   - x_1 \frac{d f_{n+1}(x_3)}{d x_3}.
  \]  
\end{lemma}
\begin{proof}
  By Lemma~\ref{lem:P_i}, the polynomials $P_i^{LR^{-(n+2)}}$ $(i=1, 2, 3)$ turn out to be
  \begin{align*}
    P_1^{LR^{-(n+2)}} &= x_3,\\
    P_2^{LR^{-(n+2)}} &= x_2 f_{-(n+1)}(x_3) -x_1f_{-(n+2)}(x_3) = -x_2 f_{n+1}(x_3) +x_1f_{n+2}(x_3),\\
    P_3^{LR^{-(n+2)}} &= x_2 f_{-n}(x_3) -x_1f_{-(n+1)}(x_3)) = -x_2 f_n(x_3) + x_1f_{n+1}(x_3).
  \end{align*}
  It follows from the general formula~\eqref{eq:adT_general} that
\begin{align*}
  \BT_{M_n, \lambda}(\rho)
  &= 3 - \frac{\partial P_1^{LR^{-(n+2)}}}{\partial x_1} - \frac{\partial P_2^{LR^{-(n+2)}}}{\partial x_2} - \frac{\partial P_3^{LR^{-(n+2)}}}{\partial x_3} \\
  &= 3 - \left( -f_{n+1}(x_3) - x_2 \frac{d f_n(x_3)}{d x_3}   + x_1  \frac{d f_{n+1}(x_3)}{d x_3} \right) \\
  &= 3 + f_{n+1}(x_3) + x_2 \frac{d f_{n}(x_3)}{d x_3} - x_1 \frac{d f_{n+1}(x_3)}{d x_3}.
\end{align*}
\end{proof}
The variables $x_i$ in our situation are expressed as
\begin{align*}
x_1 &= \tr \rho(\beta) = y, \\
x_2 &= \tr \rho(\beta') = \tr \rho(\beta^{-1} \mu \beta \mu^{-1}) = \tr \rho(\alpha \beta \alpha^{-1} \beta^{-1}) = x^2 + y^2 + z^2 - xyz - 2,\\
x_3 &= \tr \rho(\beta \beta') =  \tr \rho(\mu \beta \mu^{-1}) = \tr \rho(\beta) = y.
\end{align*}
We can regard the adjoint Reidemeister torsion $\BT_{M_n, \lambda}(\rho)$ in Lemma~\ref{lem:adT} as a function in $y$
on our geometric components.
\begin{proposition}
  \label{prop:torsionOPTB}
  If $x \not= 0$, then the variable $x_2$ satisfies
  \[x_2 = x^2 + y^2 + z^2 - xyz - 2
  = \frac{yz}{x} = y \frac{f_{n}(y)-y}{f_{n-1}(y)-1}.\]

  Hence we can express the adjoint Reidemeister torsion
  $\BT_{M_n, \lambda}(\rho)$ in Lemma~\ref{lem:adT} as a function in $y$ as follows:
  \[
  \BT_{M_n, \lambda}(\rho)
  = 3 + f_{n+1}(y)
  + y \frac{f_n(y)-y}{f_{n-1}(y)-1}  \frac{d f_{n}(y)}{d y}
  - y  \frac{d f_{n+1}(y)}{d y}.\]
\end{proposition}
\begin{proof}
  By Eqs.~\eqref{eq:phi1_BP} and~\eqref{eq:phi2_BP} the variable
  $x_2 = x^2 + y^2 + z^2 - xyz - 2$ turns into
  \begin{equation}
    \label{eq:phi1phi2intox2}
    1-f_{n-1}(y) + y^2 + z^2 - y(y-f_n(y)) -2
    = yf_n(y) - f_{n-1}(y) -1 + z^2.
  \end{equation}
  By Eq.~\eqref{eq:phi3_BP} we have $yf_n(y) - f_{n-1}(y) -1 = zf_n(y) /x$. We can rewrite \eqref{eq:phi1phi2intox2} as
  \[
  \frac{zf_n(y)}{x} + z^2
  = \frac{z(f_n(y) + xz)}{x}
  = \frac{yz}{x}
  \]
  using Eq.~\eqref{eq:phi2_BP} again.
\end{proof}
We close this section with the explicit form of adjoint Reidemeister torsion on the extra components $L$.
Suppose $x=y=0$ (this occurs only when $n=4k+2$, $k \in \BZ$). Then we have the slope $\ell = m^{-4}$
and the local coordinate $x_2 = z^2 -2 = m^2+m^{-2}$.
By Lemma~\ref{lem:adT} and the subsequent explanation of the variables,
the adjoint Reidemeister torsion $\BT_{M_n, \lambda}(\rho)$ on the extra component $L$ is expressed as 
\begin{equation}
  \label{eq:torsion_extra}
  \BT_{M_n, \lambda}(\rho)
  = 3 + f_{n+1}(y) + (m^2+m^{-2}) \frac{d f_{n}(y)}{d y}   - y \frac{d f_{n+1}(y)}{d y}
\end{equation}
at $y=0$.
Moreover the derivative of $f_k(y)$ is also expressed as follows.
\begin{lemma}
  \label{lem:diff_f}
  We have
  \[
  \frac{df_k(y)}{dy} = \frac{(k-1) f_{k+1}(y) - (k+1) f_{k-1}(y)}{y^2-4}.
  \]
\end{lemma}
The value of $f_n$ at $y=0$ is expressed as follows.
\begin{lemma}[{\cite[Lemma~4.3]{Baker-Petersen}}]
  \label{lem:fn0}
  $f_{2k}(0) = 0$ and $f_{2k+1}(0) = (-1)^k$.
\end{lemma}

From Lemma~\ref{lem:diff_f} and $f_{n+1}(0) = f_{4k+3}(0) = -1$ the derivatives of \eqref{eq:torsion_extra} turn out to be 
\begin{align*}
  \left. \frac{df_n(y)}{dy} \right|_{y=0}
  &= \left. \frac{(n-1) f_{n+1}(y) - (n+1) f_{n-1}(y)}{y^2-4} \right|_{y=0}
  = \frac{(n-1)(-1) - (n+1)(1)}{-4} = \frac{n}{2}, \\
  \left. \frac{df_{n+1}(y)}{dy} \right|_{y=0}
  &= \left. \frac{n f_{n+2}(y) - (n+2) f_{n}(y)}{y^2-4} \right|_{y=0}
  = 0.
\end{align*}
Hence we can rewrite Eq.~\eqref{eq:torsion_extra} as
\[\BT_{M_n, \lambda}(\rho) = 2 + (m^2 + m^{-2})\frac{n}{2}.\]

\begin{proposition}
  \label{prop:torsion_L}
  When $n \equiv 2 \pmod{4}$, we can express the adjoint Reidemeister torsion $\BT_{M_n, \gamma}(\rho)$ with respect to $\gamma=\mu^{p}\lambda^{q}$ on the extra component $L$ as the following function in $m$:
  \begin{equation}
  \label{eq:torsion_pq}
  \BT_{M_n, \gamma}(\rho) = \left( - \frac{p}{4}+ q \right) \left( 2 + (m^2 + m^{-2})\frac{n}{2}\right).
  \end{equation}

\end{proposition}
\begin{proof}
  This follows from the curve change formula (see~\cite[Theorem~4.1]{Porti}),
  which says 
  \begin{equation}
    \label{eq:curve_change}
    \BT_{M_n,\gamma}(\rho) = \frac{d \log (m^p \ell^q)}{d \log \ell} \BT_{M_n, \lambda}(\rho).
  \end{equation}
  Note that $\tr \rho(\mu^p \lambda^q) = m^p \ell^q + m^{-p}\ell^{-q} = m^{p-4q} + m^{-p+4q}$. 
\end{proof}

\section{Eigenvalues for the meridian $\mu$ and the longitude $\lambda$}
The adjoint Reidemeister torsion is defined by the pair of a hyperbolic $3$-manifold and a slope $\gamma$ in
the boundary.
We need to find the eigenvalues for the meridian $\mu$ and the longitude $\lambda$ to consider
the adjoint Reidemeister torsion with respect to an arbitrary slope $\gamma$.

If $\rho$ is irreducible, then $\rho(\mu) = \rho(\beta \alpha)$ and
$\rho(\beta)$ are also not conjugate to upper triangular matrices simultaneously.
For an irreducible $\SL_2(\BC)$-representation $\rho$, 
up to conjugation, we may assume that 
\[
\rho(\mu) = \rho(\beta \alpha) = \left[ \begin{array}{cc}
m & 1 \\
0 & 1/m \end{array} \right]
\quad \text{and} \quad 
\rho(\beta) = \left[ \begin{array}{cc}
b & 0 \\
* & 1/b\end{array} \right].
\]
The trace of $\rho(\alpha)=\rho(\beta)^{-1}\rho(\mu)$ equals to $m/b-*+b/m$.
Using the variable $x = \tr\rho(\alpha)$, we express $\rho(\beta)$ as
\[\rho(\beta) = \left[ \begin{array}{cc}
b & 0 \\
-x+b/m+m/b& 1/b\end{array} \right].\]
It follows from the irreducibility of $\rho$ that $-x+b/m+m/b \not =0$.

Since  $\lambda = \alpha \beta \alpha^{-1} \beta \alpha \beta^{-1} \alpha^{-1} \beta^{-1}$ and
$\alpha = \beta^{-1} \mu$, by a direct calculation we have
\[
\rho(\lambda) = \left[ \begin{array}{cc}
\ell & \cdots \\
bm(b/m+m/b-x)(**) & \ell^{-1} \end{array} \right]\]
where 
\begin{align*}
** &= (b^4 m^2 x-b^3 m^3 x^2-b^3 m^3-b^3 m x^2-b^3 m+b^2 m^4 x+b^2 m^2 x^3 \\
&\quad + \, 2 b^2 m^2 x+b^2 x - b m^3 x^2-b m^3-b m x^2-b m+m^2 x)/(b^3m^4).
\end{align*}
and $\ell =- m(**)+ m^{-4} - x y m^{-3} + x^2 m^{-2}$.
Since $\rho(\mu)$ and $\rho(\lambda)$ are commutative, the $(2, 1)$-entry of $\rho(\lambda)$ equals to $0$.
It follows from $-x+b/m+m/b \not =0$ that $**=0$.
Thus we get $\ell= m^{-4} - x y m^{-3} + x^2 m^{-2}.$
Similarly, from the $(2,2)$-entry of $\rho(\lambda)$ we get $\ell^{-1} = m^{4} - x y m^{3} + x^2 m^{2}.$

We have seen that the variables $x$ and $z$ can be regarded as functions in $y$.
We also describe the eigenvalues $m$ and $\ell$ as functions in $y$.

On the geometric component the functions $x$ and $z$ are expressed as  
\[x = \ve \sqrt{1-f_{n-1}(y)}, \, z = \ve \frac{y - f_n(y)}{\sqrt{1-f_{n-1}(y)}},\] 
where $\ve = \pm 1$ and $f_{n-1}(y) \not= 1$.
We can also assume $f_{n}(y) \not= y$ since we consider generic points on $X(\pi_1(M_n))$ (actually it is enough to assume $k_n(y) \not= 0$, see Subsection~\ref{subsec:newdefpolys}).
Then \[x = z \frac{1-f_{n-1}(y)}{y - f_n(y)} = (m+m^{-1})\frac{1-f_{n-1}(y)}{y - f_n(y)}.\]

We have 
\begin{align*}
\ell m^2 &= m^{-2} - x y m^{-1}+ x^2 \\
&= m^{-2} - (1+m^{-2})\frac{y - y f_{n-1}(y)}{y - f_n(y)}+ 1- f_{n-1}(y) \\
&= m^{-2} ( 1 -  \frac{y - y f_{n-1}(y)}{y - f_n(y)} ) + 1- f_{n-1}(y) - \frac{y - y f_{n-1}(y)}{y - f_n(y)}.
\end{align*}
Hence we can regard the eigenvalues $m$ and $\ell$ as the following functions in $y$:
\[\ell = u(y) m^{-2}+ v(y) m^{-4}, \quad m^2+m^{-2} = s(y)\]
where $u(y)$, $v(y)$ and $s(y)$ are 
\[
u(y) = -  \frac{(1- f_{n-1}(y)) f_{n}(y)}{y - f_n(y)} ,
\quad v(y) = \frac{f_{n-2}(y)}{y - f_n(y)},
\quad s(y) = \frac{(y - f_n(y))^2}{1-f_{n-1}(y)} - 2 (=z^2-2).
\]
We can extend $u(y)$ and $v(y)$ to the case of $y = 0$ by setting
$u(0)=0$ and $v(0)=1$.

Similarly $\ell^{-1} = u m^{2}+ v m^{4}$. 
We can rewrite $\ell\ell^{-1}=1$ as $u^2 + v^2 + s u v = 1$. 

\section{Relation between adjoint torsion and character variety}
We evaluate the adjoint torsion $\BT_{M, \gamma}(\rho)$ for any slope $\gamma=\mu^p\lambda^q$.
The adjoint torsion $\BT_{M, \gamma}(\rho)$ for a slope $\gamma=\mu^p\lambda^q$
is derived from $\BT_{M, \lambda}(\rho)$ in Proposition~\ref{prop:torsionOPTB} and
the curve change formula of the adjoint Reidemeister torsion.
The purpose of this section is to give an explicit form the adjoint torsion by
the Jacobian determinant of functions on the geometric component $D$ in the character variety $X(\pi_1(M_n))$ following~\cite{Yoon}.

\subsection{Defining polynomials of the character variety}
\label{subsec:newdefpolys}
We set function $E_\gamma(m,y)$ for a slope $\gamma=\mu^p\lambda^q$ on $X(\pi_1(M_n))$ as
\[E_\gamma(m,y):=m^p \ell^q = m^{p}  (u(y) m^{-2}+ v(y) m^{-4})^q.\]
Note that a solution of $t^2 - (\tr_{\gamma}) t + 1=0$ is non-constant
since the trace function $\tr_\gamma$ is non-constant on $D$ which is a norm curve in $X(\pi_1(M_n))$.

The functions $u(y)$ and $v(y)$ can be rewritten as
\[u(y) = -\frac{j_n(y) \ell^2_n(y)}{k_n(y)}, \quad v(y)=1-y\frac{\ell_n(y)}{k_n(y)}.\]
We regard $u(y)$ and $v(y)$ as being defined under the assumption that $k_n(y) \not= 0$ on $X(\pi_1(M_n))$.
We also rewrite the defining polynomials of $X(\pi_1(M_n))$ to describe $z = m + m^{-1}$ as an implicit function in $y$ under the assumption that $k_n(y) \not= 0$.

\begin{lemma}
  \label{lemma:new_equalities_X}
  We can regard $X(\pi_1(M_n)) \setminus \{k_n(y)=0\}$ as the set of points satisfying
  \begin{align}
    z^2 \ell_n(y) + h_n(y) k_n^2(y) &= 0, \label{eq:phi1} \\
    x k_n(y) -z \ell_n(y) &= 0, \label{eq:phi3} \\
     k_n(y) &\not= 0. \label{eq:kne0}
  \end{align}
\end{lemma}
\begin{proof}
  First we show $\hat \ell_n(y) \not= 0$ on the set of points satisfying~\eqref{eq:phi1}, \eqref{eq:phi3} and~\eqref{eq:kne0}. Suppose $\hat \ell_n(y)=0$. Then $h_n(y)=0$ or $k_n(y)=0$ by~\eqref{eq:phi1} which is a contradiction to that neither of $h_n(y)$ nor $k_n(y)$ has any common roots with $\hat\ell_n(y)$ by Lemma~\ref{lemma:gcm_jlhk}.

  We denote by $\varphi_1$, $\varphi_2$, $\varphi_3$ the defining polynomials in Eqs.~\eqref{eq:phi1_hell}, \eqref{eq:phi2_hk} and~\eqref{eq:phi3_jkell} respectively.
  Write $\varphi_3 = j_n(y)\varphi'_3$. The character variety $X(\pi_1(M_n))$ is the vanishing set of the ideal $(\varphi_1, \varphi_2, \varphi_3)$.
  Here the symbol $(\varphi_1, \varphi_2, \varphi_3)$ stands for the ideal generated by $\varphi_1$, $\varphi_2$ and $\varphi_3$.
  According to Eqs.~\eqref{eq:phi1_hell}--~\eqref{eq:phi3_jkell}, $X(\pi_1(M_n))$ is the union the vanishing set of the ideal $(\varphi_1, \varphi_2, \varphi_3')$ and that of $(\varphi_1, \varphi_2, j_n(y))$.
  It was shown in~\cite[Proposition~5.25]{Baker-Petersen} that the vanishing set of $(\varphi_1, \varphi_2, j_n(y))$ is contained in that of $(\varphi_1, \varphi_2, \varphi'_3)$. We can regard $X(\pi_1(M_n))$ as the vanishing set of the ideal $(\varphi_1, \varphi_2, \varphi'_3)$.
  
  Set $\varphi'_1=z \ell_n(y) + h_n(y)k^2_n(y)$. 
  Under the assumption $k_n(y) \not= 0$, each point of $X(\pi_1(M_n))$ is contained in the vanishing set of the ideal $(\varphi'_1, \varphi'_3)$ since we can rewrite $\varphi_2 = 0$ as $\varphi'_1=0$ by $\varphi_3=0$ on $X(\pi_1(M_n))$.

  On the other hand, we can rewrite $\varphi'_1=0$ as $\varphi_2=0$ on the vanishing set of the ideal $(\varphi'_1, \varphi'_3)$ under the assumption $k_n(y) \not= 0$.
  When $n \not\equiv 2 \pmod{4}$, we can rewrite $\ell_n(y) \varphi_1= 0$ by $\varphi'_3=0$ on the vanishing set of $(\varphi'_1, \varphi'_3)$ since $\ell_n(y) = \hat \ell_n(y) \not= 0$.
  When $n \equiv 2 \pmod{4}$, we can rewrite $\ell_n(y) \varphi_1= 0$ by $\varphi'_3=0$ on the vanishing set of $(\varphi'_1, \varphi'_3)$ under the assumption $y \not= 0$ since $\ell_n(y)=y\hat\ell_n(y)$. If $y=0$, then $\varphi'_3=0$ implies $x=0$ under the assumption $k_n(y) \not= 0$. These points $(0, 0, z)$ lie in the line $L$ of $X(\pi_1(M_n))$.
  Hence we can regard the vanishing set of $(\varphi'_1, \varphi'_3)$ is contained in $X(\pi_1(M_n))$ under the assumption $k_n(y) \not= 0$.
\end{proof}

\begin{remark}
  The points satisfying $k_n(y)=0$ on $X(\pi_1(M_n))$ are given by
  $( \pm \sqrt{2-y^2}, y, 0)$ where $y=2\cos(2k-1)\pi/(n+2)$ for $1 \leq k \leq n+1$.
\end{remark}

We define function $G(m, y)$ as
\[
G(m, y):= (m+m^{-1})^2 \ell_n(y) + h_n(y) k_n^2(y)
\]
also set $F(m, y)= m^2 + m^{-2} - s(y)$.

\begin{remark}
  We have the following factorization of $s(y)$ and relation between $G(m, y)$ and $F(m, y)$.
  \begin{itemize}
  \item
    We can rewrite $s(y)$ as
    \[\frac{(y-f_n(y))^2}{1-f_{n-1}(y)} -2
    = -\frac{\hat h_n(y)k_n^2(y)}{\hat \ell_n(y)} -2.\]
  \item
    The function $G(m, y)$ satisfies that $G(m, y) = \ell_n(y) F(m, y)$ on $\hat \ell_n(y) \not= 0$.
  \end{itemize}
\end{remark}
We also touch the values of $k_n$, $\hat h_n$ and $\hat \ell_n$ at $y=0$ which are needed later.
These values can be found in~\cite[Proof of Proposition~5.35]{Baker-Petersen}.
\begin{lemma}
  \label{lem:khell0}
  If $n=4k+2$, then $k_n(0)$, $\hat h_n(0)$ and $\hat \ell_n(0)$ are expressed as 
  \[
    k_n(0) = (-1)^{k+1}2, \quad
    \hat h_n(0) = (-1)^{k-1}k \quad \text{and} \quad
    \hat \ell_n(0) = (-1)^k \frac{n}{2}.
  \]
\end{lemma}
\begin{proof}
  By Lemmas~\ref{lem:diff_f} and~\ref{lem:fn0} it holds that
  \[\left.\frac{f_{2l}(y)}{y}\right|_{y=0}=\left.\frac{d}{dy} f_{2l}(y) \right|_{y=0}=l (-1)^{l-1}.\]
  The values of $k_n$, $\hat h_n$ and $\hat \ell_n$ at $0$ follow from the definitions.
\end{proof}

\subsection{Jacobian determinant and adjoint Reidemeister torsion}
We show the equality between the adjoint Reidemeister torsion for a slope $\gamma$ and
the Jacobian determinant given by the pair of functions $G$ and $E_\gamma$ in the variables $m$ and $y$.
\begin{proposition}
  \label{prop:torsion_jabobian}
  Set a slope $\gamma = \mu^p \lambda^q$. Then it holds 
  on $X(\pi_1(M_n)) \setminus \{k_n(y)\hat h_n(y)=0\}$ that
  \begin{equation}
    \label{eq:Jacobian_torsion}
  \frac{\partial(G, E_\gamma)}{\partial(m,y)} 
  = - \frac{2 E_\gamma}{m} k_n(y) \BT_{M_n, \gamma}(\rho).
  \end{equation}
  Here $\frac{\partial(G, E_\gamma)}{\partial(m,y)}$ denotes
  the Jacobian determinant
  $\det \begin{bmatrix}
    \frac{\partial G}{\partial m} & \frac{\partial G}{\partial y} \\
    \frac{\partial E}{\partial m} & \frac{\partial E}{\partial y} 
  \end{bmatrix}$.  
\end{proposition}
\begin{remark}
  The assumption $X(\pi_1(M_n)) \setminus \{k_n(y) \hat h_n(y)=0\}$ means that
  $D \setminus \{f_n(y)-y=0\}$ when $n \not \equiv 2 \pmod{4}$
  or $D \setminus \{f_n(y)-y=0\} \cup L$ when $n \equiv 2 \pmod{4}$
  since $f_n(y) - y$ equals $k_n(y) \hat h_n(y)$ when $n \not \equiv 2 \pmod{4}$ or $k_n(y)y\hat h_n(y)$ when $n \equiv 2 \pmod{4}$.
\end{remark}

Proposition~\ref{prop:torsion_jabobian} follows from the next lemmas
and the curve change formula~\eqref{eq:curve_change} of the adjoint Reidemeister torsion (see~\cite[Theorem~4.1]{Porti}).
\begin{lemma}
  \label{lemT2}
  On $D \setminus \{f_n(y) - y = 0\}$, 
  \[\frac{\partial(F, E_\gamma)}{\partial(m,y)}
  =-\frac{2E_{\ell}}{m}\left(p\frac{\ell}{m}\frac{dm}{d\ell}+q\right)
  \frac{2vs'+2u'+v's}{v}.
  \]  
\end{lemma}
\begin{proof}
  By $F=0$ we have $0 = dF= \frac{\partial F}{\partial m}  dm + \frac{\partial F}{\partial y} dy$. This implies that
  \begin{align*}
    \frac{\partial(F, E_\gamma)}{\partial(m, y)}
    &= - \frac{\partial F}{\partial y} \frac{\partial E_\gamma}{\partial m}
    +  \frac{\partial F}{\partial m}  \frac{\partial E_\gamma}{\partial y}\\
    &= - \frac{\partial F}{\partial y}  \frac{\partial E_\gamma}{\partial m}
    - \frac{\partial F}{\partial y}  \frac{d y}{d m} \frac{\partial E_\gamma}{\partial y}\\
    &= - \frac{\partial F}{\partial y}
    \left( \frac{\partial E_\gamma}{\partial m} + \frac{\partial E_\gamma}{\partial y} \frac{d y}{d m} \right)\\
    \intertext{by $d E_\gamma = \frac{\partial E_\gamma}{\partial m} dm + \frac{\partial E_\gamma}{\partial y} dy = \frac{\partial E_\gamma}{\partial m} dm + \frac{\partial E_\gamma}{\partial \ell} d\ell $}
    &= - \frac{\partial F}{\partial y}
    \left( \frac{\partial E_\gamma}{\partial m} + \frac{\partial E_\gamma}{\partial \ell} \frac{d\ell}{dm} \right)\\
    &= - \frac{\partial F}{\partial y}
    \left( p\frac{E_\gamma}{m}\frac{dm}{d\ell} + q\frac{E_\gamma}{\ell}  \right)\frac{d\ell}{dm} \\
    &= - \frac{E_\gamma}{\ell}
    \left( p\frac{\ell}{m}\frac{dm}{d\ell} + q\right)\frac{\partial F}{\partial y}\frac{d\ell}{dm}.
  \end{align*}

  Since $\ell = u(y) m^{-2}+ v(y) m^{-4}$ and $d\ell = \frac{\partial \ell}{\partial m}dm + \frac{\partial \ell}{\partial y}dy$,
  we can rewrite $- \frac{\partial F}{\partial y}\frac{d \ell}{d m}$ as
  \begin{align*}
    - \frac{\partial F}{\partial y}  \frac{d \ell}{d m}
    &= - \frac{\partial F}{\partial y}  \left(-2m^{-3} u - 4  m^{-5} v +(u' m^{-2}+v' m^{-4}) \frac{dy}{dm} \right) \\
    &=  (2m^{-3} u +4  m^{-5} v) \frac{\partial F}{\partial y} + (u' m^{-2}+v' m^{-4}) \frac{\partial F}{\partial m} \\
    &=  (2m^{-3} u +4  m^{-5} v) (-s')+ (u' m^{-2}+v' m^{-4}) (2m-2m^{-3}) \\
    &= 2m^{-5} \left( u' m^4 + (v' - us')m^2 - 2vs'- u' -v' m^{-2} \right) \\
    &= 2m^{-5} \left( u' (s m^2-1) + (v' - us')m^2 - 2vs' - u' - v' (s-m^2) \right) \\
    &= 2m^{-5} \left( (u's -  us'+ 2v')m^2 - 2vs'- 2u'-v's \right).
  \end{align*}

  By taking the derivative of  $u^2+v^2+uvs=1$ we have $2uu'+2vv'+u'vs+uv's+uvs'=0$.
  This implies that $u's -  us'+ 2v' = - \frac{u}{v} (2vs' + 2u'+v's)$ and so 
  \[
  - \frac{\partial F}{\partial y}  \frac{d \ell}{d m}
  =  - 2m^{-5}\frac{2vs' + 2u'+v's}{v}(um^2 +v)
  =  - 2 \ell m^{-1} \,  \frac{2vs' + 2u'+v's}{v}.
  \]
  Hence we obtain 
  \begin{align*}
    \frac{\partial(F, E_\gamma)}{\partial(m, y)}
    &= - \frac{E_\gamma}{\ell}  \left( p \frac{\ell}{m} \frac{d m}{d \ell}+ q \right) \frac{\partial F}{\partial y} \frac{d \ell}{d m}  \\
    &= - 2 \frac{E_\gamma}{m}  \left( p \frac{\ell}{m} \frac{dm}{d\ell}+ q \right) \frac{2vs' + 2u'+v's}{v}.
  \end{align*} 
\end{proof}
\begin{lemma}
  \label{lemT}
  On $D \setminus \{f_n(y)-y = 0\}$,
  \[
  2vs' + 2u'+v's
  = \frac{f_{n-2}(y)}{1-f_{n-1}(y)}  \BT_{M_n, \lambda}(\rho)
  = \frac{v(y)k_n(y)}{\ell_n(y)} \BT_{M_n, \lambda}(\rho).
  \]
\end{lemma}
\begin{proof}
  Write $y = b+b^{-1}$. 
  Then $f_k(y) = \frac{b^k - b^{-k}}{ b - b^{-1} }$. 
  By a direct calculation, using $\frac{d g(y)}{d y} = \frac{d g(y)}{d b} / \frac{d y}{d b}$,
  we can express $2vs' + 2u' + v's$ as
  \begin{align*}
    2vs' + 2u'+v's &=
    \frac{b^{-n-1} (b^n+b^2) }{(b-1)^2 (b+1)^2 (b^n+1)^2} (-n b^{2 n}-4 b^{2 n}+6 b^{n+2}-n b^{n+4}-4 b^{n+4} \\
    &\quad  + \, 6 b^{2 n+2}+n b^{2 n+4}-4 b^{2 n+4}+2 b^{3 n+2}+n b^n-4 b^n+2 b^2),\\
    \BT_{M_n, \lambda}(\rho)
    &= -\frac{b^{-n} }{(b-1)^2 (b+1)^2 (b^n+1)} (-n b^{2 n}-4 b^{2 n}+6 b^{n+2}-n b^{n+4}-4 b^{n+4} \\
    &\quad  + \, 6 b^{2 n+2}+n b^{2 n+4}-4 b^{2 n+4}+2 b^{3 n+2}+n b^n-4 b^n+2 b^2).
  \end{align*}
  The lemma follows from
  \[
    \frac{f_{n-2}(y)}{1-f_{n-1}(y)} = -\frac{b^n + b^2}{b (b^n+1)} \quad
    \text{and} \quad
    \frac{f_{n-2}(y)}{1-f_{n-1}(y)} = \frac{v(y)(y-f_n(y))}{1-f_{n-1}(y)}
    = \frac{v(y)k_n(y)}{\ell_n(y)}.
  \]
\end{proof}
We turn to the proof of Proposition~\ref{prop:torsion_jabobian}.
\begin{proof}[Proof of Proposition~\ref{prop:torsion_jabobian}]
  The Jacobian determinant in the left hand side turns into
  \begin{equation}
    \label{eq:jac_tor}
  \frac{\partial(\ell_n F, E_\gamma)}{\partial(m,y)}
  =\ell_n(y) \frac{\partial(F, E_\gamma)}{\partial(m,y)}
  -\ell_n'(y)F(m, y)\frac{\partial E_\gamma}{\partial m}.
  \end{equation}
  We have $F(m, y)=0$ on $D \setminus \{f_n(y)-y = 0\}$.
  It follows from Lemmas~\ref{lemT2} and~\ref{lemT} that
  \begin{align*}
    \frac{\partial(G, E_\gamma)}{\partial(m,y)}
    &= -\frac{2E_\gamma}{m}\left(p\frac{\ell}{m}\frac{dm}{d\ell}+q\right)
    k_n(y) \BT_{M_n, \lambda}(\rho) \\
    &= -\frac{2E_\gamma}{m} k_n(y) \BT_{M_n, \gamma}(\rho).
  \end{align*}
  Note that we use the curve change formula~\eqref{eq:curve_change}, that is, 
  \[\BT_{M_n, \gamma}(\rho) = \frac{d \log (m^p \ell^q)}{d \log \ell} \BT_{M_n, \lambda}(\rho)
  = \left( p \frac{\ell}{m}  \frac{dm}{d\ell}+ q \right) \BT_{M_n, \lambda}(\rho)\]
  in the last equality.

  Next we consider the component $L$ in the case of $n \equiv 2 \pmod{4}$.
  The function $\ell_n(y)$ turns into $y \hat\ell_n(y)$. It holds that $x=0$, $y=0$ and $m^4\ell = 1$ on $L$.
  Together with
  \[\frac{\partial E_\gamma}{\partial m} = \frac{\partial}{\partial m} m^p \ell^q = p \frac{E\gamma}{m}+q\frac{E_\gamma}{\ell} \frac{d\ell}{dm}=\frac{E\gamma}{\ell}\left(p \frac{\ell}{m}\frac{dm}{d\ell}+q\right)\frac{d\ell}{dm},\]
  we have 
  \[
  \frac{\partial(G, E_\gamma)}{\partial(m,y)}
  = - \left.\frac{E_\gamma}{\ell} \frac{d\ell}{dm} \left(p\frac{\ell}{m}\frac{dm}{d\ell}+q\right) \hat\ell_n(y)F(m, y)\right|_{y=0}.
  \]
  It follows from Remark~\ref{lem:khell0} that 
  \[\hat \ell_n(0) = (-1)^{k}\frac{n}{2} \quad \text{and} \quad
  -\frac{k_n^2(0) \hat h_n(0)}{\hat \ell_n(0)} = 4\left( \frac{1}{2} - \frac{1}{n}\right)\]
  where $n=2(2k+1)$.
  Since $m^4\ell=1$ on $L$ and $F(m, y)= m^2 + m^{-2} - s(y) = m^2 + m^{-2} + (k^2_n(y) \hat h_n(y))/ \hat \ell_n(y) + 2$, 
  the Jacobian determinant $\frac{\partial(G, E_\gamma)}{\partial(m,y)}$ 
  is expressed as
  \begin{align*}
    \frac{\partial(G, E_\gamma)}{\partial(m, y)}
    &= - \frac{E_\gamma}{m^{-4}} (-4m^{-5})
    \left.\left( p \frac{\ell}{m} \frac{dm}{d\ell}+ q \right)\right|_{y=0}
    (-1)^k \frac{n}{2}
    \left( m^2 + m^{-2} + \frac{4}{n}\right) \\
    &=  - \frac{2 E_\gamma}{m}
    (-1)^{k+1}2
    \left( -\frac{p}{4} + q \right)
    \left( \frac{n}{2} (m^2 + m^{-2}) + 2\right).
  \end{align*} 
  By $k_n(0)=(-1)^{k+1}2$ from Remark~\ref{lem:khell0} and
  Proposition~\ref{prop:torsion_L},
  we have
  \[\frac{\partial(G, E_\gamma)}{\partial(m, y)} = -\frac{2E_\gamma}{m}k_n(0) \BT_{M_n, \gamma}(\rho).\]
\end{proof}

\section{Checking conjecture}
We will show the following main theorem
which gives infinitely many new supporting examples to the conjecture by~\cite{Gang}.
\begin{theorem}
  \label{thm:main}
  The conjecture by~\cite{Gang} is true 
  for every hyperbolic once-punctured torus bundle $M_n$ with tunnel number one.
\end{theorem}
Since $\BT_{M_n, \gamma^{-1}}(\rho) = - \BT_{M_n, \gamma}(\rho)$
we can assume that $q \ge 0$.
The proof will be divided into the cases $n \not\equiv 2$ and $n \equiv 2 \pmod{4}$.
Subsections~\ref{subsec:nnot2mod4} and~\ref{subsec:n2mod4} deal with the cases of $n \not\equiv 2$ and $n \equiv 2 \pmod{4}$ respectively.
\subsection{$n \not\equiv 2 \pmod{4}$}
\label{subsec:nnot2mod4}
The character variety has only geometric component $D$. The function $\ell_n(y)$ and $h_n(y)$ satisfy $\ell_n(y) = \hat \ell_n(y)$ and $h_n(y) = \hat h_n(y)$.
We have $m^2 + m^{-2} - s(y) = 0$ from $G(m, y) =0$ since $G(m, y) = \ell_n(y)F(m, y)$ and $\ell_n(y) \not= 0$.

\subsubsection{$p$ is even}
\label{subsec:check_X0_I}
By $m^2 + m^{-2}=s$ and
$f_k(s)=(m^{2k}-m^{-2k})/(m^2-m^{-2})$
we have $m^{2k} = f_{k}(s) m^2 - f_{k-1}(s)$.
Let $\delta_r(y) = \sum_{k=0}^q {q \choose k} u^k(y)  v^{q-k}(y) f_{k+r}(s(y)).$
Recall that $s$ is a function in $y$.
Hence $E_\gamma (m,y)$ turns out to be 
\[
E_\gamma (m,y)
=  m^p \ell^q = m^{p-4q}  (u m^{2}+ v)^q
= \sum_{k=0}^q {q \choose k} u^k v^{q-k} m^{2k+p-4q}
= g(y) m^2 - h(y)
\] 
where $g(y) = \delta_{p/2-2q}(y)$ and $h(y) = \delta_{p/2-2q-1}(y)$. 

Similarly, $m^{-p}\ell^{-q} = g(y) m^{-2} - h(y)$.
For abbreviation, we use $g$ and $h$ instead of $g(y)$ and $h(y)$.
We can rewrite $(m^p\ell^q)(m^{-p}\ell^{-q})=1$ as $g^2 + h^2 - s g h =1$.
By Proposition~\ref{prop:torsion_jabobian} and Eq.~\eqref{eq:jac_tor} we can see that 
\begin{align*}
  \frac{-2 E_\gamma}{m} k_n(y) \BT_{M_n, \gamma}(\rho)
  &= \frac{\partial(G(m, y),  g m^2 - h)}{\partial(m,y)} \\
  &= \ell_n(y)\frac{\partial(F(m, y),  g m^2 - h)}{\partial(m,y)}\\
  &= \ell_n(y) \{(2m - 2 m^{-3}) (m^2 g' - h') + s' (2m g)\}.
\intertext{Then}
-E_\gamma \frac{ k_n(y) }{ \ell_n(y) } \BT_{M_n, \gamma}(\rho)
&= (m^2 - m^{-2}) (m^2 g' - h') + m^2 g s' \\
&= m^4 g' + m^2 (g s' - h') - g' + m^{-2} h' \\
&= (s m^2 -1) g' + m^2 (g s' - h') - g' + (s- m^2)h' \\
&= m^2 (s g' + g s' - 2 h') - 2 g' + s h'.
\end{align*}

By taking derivative of $g^2 + h^2 - s g h =1$ with respect to $y$,
we have $2g g' + 2h h' - (sg)'h - s g h'=0$.
This implies that $2 g' -s h' = \frac{h}{g} \left( (s g)'  -2 h' \right)$ and 
\begin{align*}
-E_\gamma \frac{ k_n(y) }{ \ell_n(y) }  \BT_{M_n, \gamma}(\rho)
&= \left(m^2 - \frac{h }{g}\right) \left( (s g)'  -2 h' \right)
= \frac{E_\gamma}{g} \left( (s g)'  -2 h' \right).
\end{align*}
We have
\[- \frac{ k_n(y) }{ \ell_n(y) } \BT_{M_n, \gamma}(\rho) = \frac{P'}{g}\] 
by putting $P(y) =  sg - 2h = g(m^2 + m^{-2}) -2h = \tr \rho(\mu^p \lambda^q)$.

Since $s = -\frac{k^2_n(y) h_n(y)}{\ell_n(y)} - 2$ and $g^2 + h^2 - s g h = 1$
it follows that 
\begin{align}
  P^2 - 4
  &= (g s - 2h)^2 - 4(g^2 + h^2 - s g h) \notag \\
  &= (s^2-4) g^2 \label{eq:P^2-4}\\
  &= h_n(y) \big[ k_n^2(y)h_n(y) + 4 \ell_n(y) \big] \left(\frac{k_n(y)}{\ell_n(y)}  g \right)^2. \notag
\end{align}
Note that any prime factor of the denominator of $g$ is a factor of $k_n(y)$ or $\ell_n(y)$
since $g$ is defined as a polynomial of a rational functions $u(y)$, $v(y)$ and $s(y)$ whose denominators are expressed as $k_n$ and $\hat \ell_n$.

So $h_n(y) \big[ k_n^2(y)h_n(y) + 4 \ell_n(y)) \big]$ is coprime with the denominator of $\frac{k_n(y)}{\ell_n(y)}g$.
This implies that the denominator of $\left(\frac{k_n}{\ell_n(y)} g \right)^2$ in \eqref{eq:P^2-4} is exactly the denominator of $P^2-4$.
Hence the denominator of $\frac{k_n(y)}{\ell_{n}(y)} g$ is exactly the denominator of $P$. 

Write $P = \frac{P_1}{P_2}$ and $-\frac{k_n(y)}{\ell_n(y)} g = \frac{R}{P_2}$ where each pair of
$\{P_1, P_2\}$ and $\{R, P_2\}$ is coprime.
Then we can express the sum of $(\BT_{M_n, \gamma}(\rho))^{-1}$ over $P=c$ as
\[
\sum_{P = c} \frac{1}{\BT_{M_n, \gamma}(\rho)}
= \sum_{P_1 - cP_2=0} \frac{\frac{R}{P_2}}{\left( \frac{P_1 - cP_2}{P_2}\right)'}
= \sum_{P_1 - cP_2=0} \frac{\frac{R}{P_2}}{\frac{(P_1 - cP_2)'}{P_2}}
= \sum_{P_1 - cP_2=0} \frac{R}{(P_1 - cP_2)'}.
\]

We now apply the Jacobi's residue theorem:
if $f$ is a non-constant polynomial with $f(0) \not= 0$ and no repeated zero and $g$ is a polynomial with $\deg g \le \deg f -2$ then 
\[
\sum_{f(z)=0} \frac{g(z)}{f'(z)} =0.
\]
(We refer the reader to~\cite[Chap.~5]{GH} for the details.)

For generic $c$, $P_1 - c P_2$ has nonzero constant coefficient and no repeated zero.
Hence we have the equality that
$\sum_{P = c} \frac{1}{\BT_{M_n, \gamma}(\rho)} =0$ if we can show that $\deg R \le \deg (P_1 - cP_2) -2$.
This is equivalent to showing that
$\deg R/P_2 \le \deg (P_1/P_2 - c) -2$, that is
$\deg \frac{k_n(y)}{\ell_n(y)} g \le \deg (P-c) -2$. 

Note that the degree of $s$ equals $n$ if $n \ge 3$ and equals $-n-2$ if $n \le -3$. 
By $P^2-4 =(s^2-4)g^2$, we have  $\deg (s g) \le \max\{\deg P,0\} = \deg (P-c)$ for generic $c$. Hence  
it holds that
\[
\deg (P-c) - \deg \frac{k_n(y)}{\ell_n(y)} g
\ge  \deg s(y)  -  \deg \frac{k_n(y)}{\ell_n(y)} 
= \deg k_n(y)h_n(y) 
=|n|-1 \ge 2,
\]
by $f_n(y)-y = k_n(y)h_n(y)$.
This proves that $\sum_{P = c} \frac{1}{\BT_{M_n, \gamma}(\rho)}=0$.

\subsubsection{$p$ is odd}
\label{subsec:check_X0_II}
In this case we can express $E_\gamma (m,y)$ as
\begin{align*}
  E_\gamma (m,y)
  &= m^{p-4q}  (u(y) m^{2}+ v(y))^q \\
  &= m^{-1} \sum_{k=0}^q {q \choose k} u^k(y)  v^{q-k}(y) m^{2k+p+1-4q} \\
  &= g m - h m^{-1}
\end{align*}
where we use the same notations $g = \delta_{(p+1)/2-2q}(y)$ and $h = \delta_{(p-1)/2-2q-1}(y)$
as in the case of an even $p$.
Similarly $m^{-p}\ell^{-q} = g m^{-1} - h m$.
By Eq.~\eqref{eq:jac_tor} the right hand in Proposition~\ref{prop:torsion_jabobian} turns out to be
\begin{align}
  -\frac{2 E_\gamma}{m} k_n(y) \BT_{M_n, \gamma}(\rho)
  &= \ell_n(y) \frac{\partial(m^2 + m^{-2} - s,  g m - h m^{-1})}{\partial(m, y)} \notag \\
  -\frac{2 E_\gamma}{m} \frac{ k_n(y) }{ \ell_n(y) } \BT_{M_n, \gamma}(\rho)
  &= (2m - 2 m^{-3}) (m g' - h' m^{-1}) + s' (g + h m^{-2}) \notag \\
  &= 2m^2 g' + s'g - 2h' + m^{-2}(s'h -2g') + 2m^{-4} h' \notag \\
  &= 2(s-m^{-2}) g' + s'g - 2h' + m^{-2}(s'h -2g') + 2(sm^{-2}-1)h' \notag \\
  &= m^{-2}(-4g' + s'h +2sh') +  2sg' +s'g -4h'. \label{eq:torsion_odd_case}
\end{align}

By taking derivative of $g^2 + h^2 - s g h =1$,
we have $2g g' + 2h h' - s'gh - s g'h - s g h'=0$.
This implies that $2sg' +s'g -4h' = - \frac {g}{h}(-4g' + s'h +2sh')$ and
we can rewrite Eq.~\eqref{eq:torsion_odd_case} as 
\begin{align*}
  -2 E_\gamma \frac{ k_n(y) }{ \ell_n(y) } \BT_{M_n, \gamma}(\rho)
  &= (m^{-1} - \frac{g}{h} m)(-4g' + s'h +2sh') \\
  &= -\frac{E_\gamma}{h} (-4g' + s'h +2sh').
\end{align*}
Since $E_\gamma = m^p \ell^q \not =0$, we have that
\begin{align}
  -\frac{ k_n(y) }{ \ell_n(y) } \BT_{M_n, \gamma}(\rho)
  &= -\frac{-4g' + s'h +2sh'}{2h} \notag \\
  &= \frac{2sg' +s'g -4h'}{2g}. \label{eq:torsion_rewritten}
\end{align}

Set $Q = (\tr \rho(\mu^p \lambda^q))^2$.  
The function $(\tr \rho(\mu^p \lambda^q))^2$ equals to
$(g-h)^2(m+m^{-1})^2 = (g-h)^2(s+2)$.
Then the fraction $Q' / (g -h)$ is expressed as
\begin{align*}
  \frac{Q'}{g -h}
  &= 2(g'-h')(s+2)+(g-h)s' \\
  &= (2sg' +s'g -4h') -(-4g' + s'h +2sh') \\
  &= (2sg' +s'g -4h') \left(1+ \frac{h}{g}\right).
\end{align*}
This implies that 
\[ \frac{Q'}{g^2 -h^2} =  \frac{2sg' +s'g -4h'}{g}.\]
Hence Eq.~\eqref{eq:torsion_rewritten} turns out to be 
\begin{equation}
  \label{eq:torsion_odd}
  -\frac{ k_n(y) }{ \ell_n(y) }  \BT_{M_n, \gamma}(\rho)  = \frac{Q'}{2(g^2 - h^2)}.
\end{equation}

Note that $\tr \rho(\mu^p \lambda^q) = (g-h)z$ where $z = \ve k_n(y)\sqrt{-\frac{h_n(y)}{\ell_n(y)}}$ and $\ve = \pm 1$.
Hence it follows that 
\[\sum_{\tr \rho(\mu^p \lambda^q)= c} \frac{1}{\BT_{M_n, \gamma}(\rho)}  =  \frac{1}{2} \sum_{Q = c^2}  \frac{1}{\BT_{M_n, \gamma}(\rho)}\]
for generic $c\in \BC$.
By Eq.~\eqref{eq:torsion_odd} 
we can rewrite $\frac{1}{2}\sum_{Q = c^2}  \frac{1}{\BT_{M_n, \gamma}(\rho)}$ as 
\[
\sum_{Q = c^2}  \frac{-k_n(y)}{\ell_n(y)} \cdot \frac{2(g^2 - h^2)} {Q'}
\]
since $z=\pm\sqrt{s(y)+2}$.
By similar arguments as in the above case with $c$, $P$ and $g$ replaced by $c^2-2$, $Q-2$ and $g^2-h^2$ respectively, we can show
\[\sum_{Q-2 = d}  \frac{-k_n(y)}{\ell_n(y)} \cdot \frac{2(g^2 - h^2)} {(Q-2)'} =0\]
for generic $d (=c^2-2) \in \BC$.
Since $g^2 + h^2 - s g h = 1$ and $Q-4 = (s+2)(g-h)^2 - 4(g^2 + h^2 - s g h) = (s-2)(g + h)^2$, 
the function $Q-2=(m^p\ell^q)^2 + (m^{-p}\ell^{-q})^2 = \tr \rho( (\mu^p \lambda^q)^2)$ satisfies 
$(Q-2)^2 - 4 = Q(Q-4) = (s^2 -4) (g^2 - h^2)^2$
which corresponds to \eqref{eq:P^2-4}.

We also have
\[(Q-2)^2 - 4 = h_n(y) \big[ h_n(y)k_n^2(y) + 4\ell_n(y) \big] \frac{k_n^2(y)}{\ell_n^2(y)}(g^2 - h^2)^2.\]
This implies that the denominator of $(Q-2)^2 - 4$ is exactly
the denominator of $\frac{k_n^2(y)}{\ell_n^2(y)}(g^2 - h^2)^2$.
Hence the denominator of $Q-2$ is exactly the denominator of $\frac{k_n(y)}{\ell_n(y)}(g^2 - h^2)$.

Write $Q-2=\frac{Q_1}{Q_2}$ and $-\frac{k_n(y)}{\ell_n(y)} \cdot 2(g^2 - h^2) = \frac{R}{Q_2}$ where
each pair of $\{Q_1, Q_2\}$ and $\{R, Q_2\}$ is coprime.
Then we can express the sum of $(\BT_{M_n, \gamma}(\rho))^{-1}$ as 
\[
\frac{1}{2}\sum_{Q-2 = d} \frac{1}{\BT_{M_n, \gamma}(\rho)}
= \sum_{Q_1 - d Q_2 =0} \frac{\frac{R}{Q_2}}{\left( \frac{Q_1 - d Q_2}{Q_2}\right)'}
= \sum_{Q_1 - d Q_2=0} \frac{\frac{R}{Q_2}}{\frac{(Q_1 - d Q_2)'}{Q_2}}
= \sum_{Q_1 - d Q_2=0} \frac{R}{(Q_1 - d Q_2)'}.
\]
This implies that  we can apply the Jacobi's residue theorem for generic $d$.
Hence we have the equality that 
$\sum_{\tr \rho(\mu^p \lambda^q) = c} \frac{1}{\BT_{M_n, \gamma}(\rho)} =0$
if we can show that $\deg R \le \deg (Q_1 - d Q_2) -2$ which is equivalent to
$\deg \frac{-k_n(y)}{\ell_n(y)} (g^2-h^2) \le \deg (Q-2-d) -2$.

By $(Q-2)^2 -4 = (s^2 - 4)(g^2 -h^2)^2$,
we have $\deg s(g^2-h^2) \le \max\{\deg (Q-2), 0\}=\deg (Q-2-d)$ for generic $d$.
Hence it holds that
\begin{align*}
&\deg (Q-2-d) - \deg \frac{-k_n(y)}{\ell_n(y)}(g^2-h^2) \\
& \ge \deg s(g^2-h^2) - \deg \frac{-k_n(y)}{\ell_n(y)}(g^2-h^2) \\
& = \deg \frac{k_n^2(y)h_n(y)}{\ell_n(y)}  -  \deg \frac{k_n(y)}{\ell(y)}\\
&= \deg k_n(y)h_n(y) = |n|-1 \ge 2
\end{align*}
by $k_n(y)h_n(y) = f_n(y) - y$.
We have thus proved
\[\sum_{\tr \rho(\mu^p \lambda^q)= c} \frac{1}{\BT_{M_n, \gamma}(\rho)}=0\] for generic $c\in \BC$ in the case of any odd $p$.

\subsection{$n \equiv 2 \pmod{4}$}
\label{subsec:n2mod4}
The character variety $X(\pi_1(M_n))$ is given by $D \cup L$ in this case.
We will see that the sums of $(\BT_{M_n, \gamma}(\rho))^{-1}$ on $D$ and $L$ are cancelled.
Propositions~\ref{prop:sumTonDevenp} and~\ref{prop:sumTonLevenp} prove Theorem~\ref{thm:main} in the case of even $p$.
Propositions~\ref{prop:sumTonDoddp} and~\ref{prop:sumTonLoddp} prove Thereom~\ref{thm:main} in the case of odd $p$.
\subsubsection{$p$ is even}
\label{subsec:evenpn2mod4}
As in the case $n \not \equiv 2 \pmod{4}$,  on the geometric component $D$ we have
\begin{align*}
  \frac{-k_n(y)}{\ell_n(y)} \BT_{M_n, \gamma}(\rho) &= \frac{P'}{g} \\
  \frac{-k_n(y)}{y \hat\ell_n(y)} \frac{g}{P'}  &= \frac{1}{\BT_{M_n, \gamma}(\rho)}
\end{align*}
where $P=sg - 2h=g(m^2+m^2) - 2h = \tr \rho(\mu^p\lambda^q)$.
We can see the rational function $P$ and $\frac{k_n(y)}{\hat \ell_n(y)} g$ has the same denominator as follows.
It follows from $g^2 + h^2 - sgh = 1$, $h_n(y)=y \hat h_n(y)$ and $\ell_n(y) = y \hat \ell_n(y)$ that
\begin{align}
  P^2 -4
  &= (gs - 2h)^2 - 4(g^2 + h^2 -sgh) \notag\\
  &= (s^2-4)g^2 \notag\\
  &= h_n(y) \big[h_n(y)k_n^2(y)+4\ell_n(y) \big] \left( \frac{k_n(y)}{\ell_n(y)}g \right)^2 \notag\\
  &= \hat h_n(y) \big[\hat h_n(y)k_n^2(y) + 4 \hat \ell_n(y) \big] \left( \frac{k_n(y)}{\hat \ell_n(y)}g \right)^2. \label{eq:P^2-4n2}
\end{align}
Write $P = \frac{P_1}{P_2}$ and $\frac{k_n(y)}{\hat \ell_n(y)} g = \frac{R_1}{R_2}$ where each pair of $\{P_1, P_2\}$ and $\{R_1, R_2\}$ is coprime.
Since any prime factor of the denominator of $g$ is a prime factor of $k_n$ or $\hat \ell_n$, 
any prime factor of the denominator $R_2$ is also a prime factor $k_n$ or $\hat \ell_n$. 
When we rewrite
\[\hat h_n(y) \big[\hat h_n(y)k_n^2(y) + 4 \hat \ell_n(y) \big] \frac{R_1^2}{R_2^2}\]
for the right hand of Eq.~\eqref{eq:P^2-4n2}, the denominator $R_2^2$ is coprime with $\hat h_n(y) \big[\hat h_n(y)k_n^2(y) + 4 \hat \ell_n(y) \big]$.
If there were any common factor of $R_2^2$ and $\hat h_n(y) \big[\hat h_n(y)k_n^2(y) + 4 \hat \ell_n(y) \big]$,
it would be $(y^2-2)^{2\alpha}$ $\alpha \geq 1$ in the case of $n \equiv 2 \pmod{8}$ since the denominator $P^2-4$ is $P_2^2$. This is impossible according to that $\hat \ell_n(y)$ is coprime with $\hat h_n(y)$.
We can set $R_2=P_2$ and $R_1=R$ such that $P_2$ and $R$ are coprime.
We can also see $R$ and $y$ are coprime which is equivalent to $g(0) \not= 0$ since $k_n(0) = (-1)^{k+1}2 \not= 0$ by Remark~\ref{lem:khell0}.

Since $u=0$ and $v=1$, that is $m^4\ell =1$, at $y=0$, $g(0)$ turns out to be 
\[
g(0) = \delta_{p/2-2q}(0)
= \sum_{k=0}^q \begin{pmatrix} q \\k \end{pmatrix} u^k(0)v^{q-k}(0) f_{k+p/2-2q}(s(0))
= f_{p/2-2q}(s(0)).
\]
\begin{lemma}
  \label{lem:s0}
  It holds that
  \[s(0)=-\frac{4}{n} \quad \text{and} \quad g(0) \not= 0.\]
\end{lemma}
\begin{proof}
  It follows from Remark~\ref{rem:param_X}~\eqref{item:intersectDL} that
  \[ s(0) = z^2 - 2 =4\left(\frac{1}{2} - \frac{1}{n}\right) -2 = -\frac{4}{n}.\]
  If $g(0)=0$, then $s(0) = m^2 + m^{-2}$ is a root of $f_{p/2-2q}(u)$, $m^{\pm 2}$ are $(p-4q)$-th roots of unity
  which means $-2 \le s(0) \le 2$.
  Since every Chebyshev polynomial is a monic polynomial with integer coefficients, 
  the roots of $f_{p/2-2q}(u)$ are algebraic integers.
  It is known that algebraic integers in $\BQ$ are integers. Hence $s(0)=-4/n$ must be $\pm 1, \pm 2$
  which can not occur by $|n| > 2$ and $n \equiv 2 \pmod{4}$.
\end{proof}

We have the sum of $(\BT_{M_n, \gamma}(\rho))^{-1}$ on the component $D$ as follows.
\begin{proposition}
  \label{prop:sumTonDevenp}
  It holds that, on the component $D$,
  \[
  \sum_{[\rho] \in \tr_{\gamma}^{-1} (c)} \frac{1}{\BT_{M_n, \gamma}(\rho)} 
  = -\frac{8}{n} \frac{f_{p/2-2q}(-4/n)}{T_{p/2-2q}(-4/n)-c}
  \]
  where $T_k(z)$ stands for $f_{k+1}(z) - f_{k-1}(z)$.
\end{proposition}
Note that $T_k(z) = a^k + a^{-k}$ if $z=a+a^{-1}$.
\begin{proof}
  We can express the sum of $(\BT_{M_n, \gamma}(\rho))^{-1}$ over $P(y)=c$ as
  \begin{align*}
    \sum_{P=c}\frac{1}{\BT_{M_n, \gamma}(\rho)}
    &= \sum_{P_1 - c P_2 =0}\frac{-\frac{R}{yP_2}}{\left(\frac{P_1-cP_2}{P_2}\right)'}
    = \sum_{P_1 - c P_2 =0}\frac{-\frac{R}{yP_2}}{\frac{(P_1-cP_2)'}{P_2}} \\
    &= \sum_{P_1 - c P_2 =0}\frac{-R}{y(P_1-cP_2)'}
    = \sum_{P_1 - c P_2 =0}\frac{-R}{[yP_1-cP_2]'}.
  \end{align*}
  As in the case $n \not \equiv 2 \pmod{4}$, the Jacobi's residue theorem implies that 
  \[
  \left.\frac{-R}{[y(P_1-cP_2)]'}\right|_{y=0} + \sum_{P_1 - c P_2 =0}\frac{-R}{[yP_1-cP_2]'}
  = \sum_{y(P_1 - c P_2) =0}\frac{-R}{[yP_1-cP_2]'} = 0.
  \]
  Hence it holds that
  \begin{align*}
    \sum_{P=c}\frac{1}{\BT_{M_n, \gamma}(\rho)}
    &= \left.\frac{R}{[y(P_1-cP_2)]'}\right|_{y=0}\\
    &= \left.\frac{R}{P_1-cP_2}\right|_{y=0} \\
    &= \left.\frac{R/P_2}{(P_1-cP_2)/P_2}\right|_{y=0} \\
    &= \left.\frac{k_n(y)}{\hat \ell_n(y)} \frac{g}{P-c}\right|_{y=0}
  \end{align*}
  We are going to carry out $P(0) = s(0)g(0)-2h(0)$.
  \begin{align*}
    h(0) &= \delta_{p/2-2q-1}(0)
    = \left. \sum_{k=0}^q \begin{pmatrix} q \\k \end{pmatrix} u^k(y)v^{q-k}(y) f_{k+p/2-2q-1}(s(y)) \right|_{y=0}
    = f_{p/2-2q-1}(s(0)).
  \end{align*}
  This implies that
  \begin{align*}
    P(0)
    &= s(0)g(0)-2h(0) \\
    &= s(0)f_{p/2-2q}(s(0)) -2 f_{p/2-2q-1}(s(0)) \\
    &= f_{p/2-2q+1}(s(0)) - f_{p/2-2q-1}(s(0)) \\
    &= T_{p/2 -2q}(s(0)).
  \end{align*}  
  It follows from Remark~\ref{lem:khell0} and Lemma~\ref{lem:s0} that 
  \[
  \sum_{P(y)=c}\frac{1}{\BT_{M_n, \gamma}(\rho)} = -\frac{4}{n} \frac{f_{p/2-2q}(-4/n)}{T_{p/2-2q}(-4/n)-c}.
  \]
  Since $z=\pm \sqrt{2+s(y)}$, we have the following equality
  \[
  \sum_{[\rho] \in \tr_{\gamma}^{-1} (c)} \frac{1}{\BT_{M_n, \gamma}(\rho)} 
  = 2\sum_{P(y)=c}\frac{1}{\BT_{M_n, \gamma}(\rho)} = -\frac{8}{n} \frac{f_{p/2-2q}(-4/n)}{T_{p/2-2q}(-4/n)-c}.
  \]
\end{proof}

On the component $L$ giving $y=0$, we have $\ell=m^{-4}$.
We set $t= m^2+m^{-2} (= z^2-2)$ since $s$ is no longer equal to $m^2+m^{-2}$ on $L$.
Then $P=\tr \rho(\mu^p \lambda^q) = m^{p-4q} + m^{-(p-4q)} = T_{p/2 - 2q}(t)$.
We express ${\BT_{M_n, \gamma}}(\rho)$ as
\[
\BT_{M_n, \gamma}(\rho)
= -\frac{p-4q}{4}\left(2 + (m^2 + m^{-2} \frac{n}{2})\right)
= -\frac{p-4q}{4}\left(2+t\frac{n}{2}\right).
\]
Since $z=\pm\sqrt{t+2}$, we have
\[
\sum_{[\rho] \in \tr_\gamma^{-1}(c)} \frac{1}{\BT_{M_n, \gamma}(\rho)}
= 2 \sum_{T_{p/2-2q}(t)=c}\frac{-4}{p-4q}\frac{1}{2+t\frac{n}{2}}
= \frac{-8/n}{p/2-2q} \sum_{T_{p/2-2q}(t)=c}\frac{1}{t+4/n}.
\]
\begin{proposition}
  \label{prop:sumTonLevenp}
  On the component $L$ it holds that 
  \[
  \sum_{[\rho] \in \tr_\gamma^{-1}(c)} \frac{1}{\BT_{M_n, \gamma}(\rho)}
  = \frac{8}{n} \frac{f_{p/2 -2q}(-4/n)}{T_{p/2-2q}(-4/n)-c}.
  \]
\end{proposition}
\begin{proof}
  It follows from Lemma~\ref{lemma:sum_t-b} that
  \begin{align*}
  \sum_{[\rho] \in \tr_\gamma^{-1}(c)} \frac{1}{\BT_{M_n, \gamma}(\rho)}
  &= \frac{-8/n}{p/2-2q} \sum_{T_{p/2-2q}(t)=c}\frac{1}{t+4/n} \\
  &= \frac{-8/n}{p/2-2q} \left( -\frac{(p/2-2q) f_{p/2-2q}(-4/n)}{T_{p/2-2q}(-4/n)-c}\right).
  \end{align*}
\end{proof}

\begin{lemma}
  \label{lemma:sum_t-b}
  \[\sum_{T_r(t)=c} \frac{1}{t-b} = -\frac{rf_r(b)}{T_r(b)-c}.\]
\end{lemma}
\begin{proof}
  We first claim that $T_r'(t)=rf_r(t)$.
  If we write $t=a + a^{-1}$, then $T'_r(t) = \frac{dT_r/da}{dt/da} = rf_r(t)$.
  By $1/r=f_r(t)/T'_r(t)$ we have
  \begin{align*}
    \frac{1}{r}\sum_{T_r(t)=c}\frac{1}{t-b}
    &= \sum_{T_r(t)=c}\frac{1}{t-b} \frac{f_r(t)}{(T_r(t)-c)'} \\
    &= \sum_{T_r(t)=c}\frac{f_r(t)}{[(t-b)(T_r(t)-c)]'}.
  \end{align*}
  Since the degree of $f_r(t)$ is $|r|-1$ and the degree of $(t-b)(T_r(t)-c)$ is $|r|+1$,
  the Jacobi's residue theorem says that
  \[\sum_{(t-b)(T_r(t)-c)=0}\frac{f_r(t)}{[(t-b)(T_r(t)-c)]'}=0.\]
  This implies that
  \begin{align*}
    \sum_{T_r(t)=c}\frac{f_r(t)}{[(t-b)(T_r(t)-c)]'}
    &= - \left.\frac{f_r(t)}{[(t-b)(T_r(t)-c)]'}\right|_{t=b} \\
    &= - \left.\frac{f_r(t)}{T_r(t)-c}\right|_{t=b} \\
    &= - \frac{f_r(b)}{T_r(b)-c}.
  \end{align*}
  Hence
  \[
  \sum_{T_r(t)=c} \frac{1}{t-b}
  = r \sum_{T_r(t)=c}\frac{f_r(t)}{[(t-b)(T_r(t)-c)]'}
  = -\frac{rf_r(b)}{T_r(b)-c}.
  \]  
\end{proof}
Combining Propositions~\ref{prop:sumTonDevenp} and~\ref{prop:sumTonLevenp}, we obtain Theorem~\ref{thm:main} in the case that $n \equiv 2 \pmod{4}$ and $p$ is even.

\subsubsection{$p$ is odd.}
On the component $D$,
we adapt the discussion in the case of $n \not\equiv 2 \pmod{4}$ as like the case of even $p$ for $n \equiv 2 \pmod{4}$.
Recall that $E_\gamma = m^p\ell^q= gm - hm^{-1}$ and $\tr \rho(\mu^p\lambda^q) = (g-h)(m+m^{-1}) = (g-h)z$.
When we set $Q = (\tr \rho(\mu^p\lambda^q))^2$, we have the following as in~Subsection~\ref{subsec:check_X0_II}:
\[
  -\frac{ k_n(y) }{ \ell_n(y) }  \BT_{M_n, \gamma}(\rho)  = \frac{Q'}{2(g^2 - h^2)}.
\]

\begin{proposition}
  \label{prop:sumTonDoddp}
  On the component $D$, we have
  \[
  \sum_{\tr \rho(\mu^p \lambda^q)= c} \frac{1}{\BT_{M_n, \gamma}(\rho)}
  = -\frac{8}{n} \frac{f_{p-4q}(-4/n)}{T_{p-4q}(-4/n)+2-c^2}.
  \]
\end{proposition}
\begin{proof}
  It follows from $z=\pm\sqrt{s(y)+2}$ that for generic $c \in \BC$.
  \begin{align*}
    \sum_{\tr \rho(\mu^p \lambda^q)= c} \frac{1}{\BT_{M_n, \gamma}(\rho)}
    &=\sum_{Q(y)=c^2}\frac{-k_n(y)}{\ell_n(y)}\frac{2(g^2-h^2)}{Q'}\\
    \intertext{putting $d=c^2-2$}
    &=\sum_{Q(y)-2=d}\frac{-k_n(y)}{\ell_n(y)}\frac{2(g^2-h^2)}{(Q-2)'},\\
    \intertext{since we can write $Q-2 = \frac{Q_1}{Q_2}$ and $-\frac{k_n(y)}{\hat \ell_n(y)}2(g^2-h^2) = \frac{R}{Q_2}$ by the same argument as in Subsection~\ref{subsec:evenpn2mod4},}
    &= \sum_{Q_1 - d Q_2 = 0}\frac{\frac{R}{yQ_2}}{\left(\frac{Q_1-d Q_2}{Q_2}\right)'} \\
    &= \sum_{Q_1 - d Q_2 = 0}\frac{R}{y(Q_1-d Q_2)'} \\
    &= \sum_{Q_1 - d Q_2 = 0}\frac{R}{[y(Q_1-d Q_2)]'}.
  \end{align*}
  The Jacobi's residue theorem implies that
  \[
  \left. \frac{R}{[y(Q_1-d Q_2)]'} \right|_{y=0} + \sum_{Q_1 - d Q_2 = 0}\frac{R}{[y(Q_1-d Q_2)]'}
  =\sum_{y(Q_1 - d Q_2) = 0}\frac{R}{[y(Q_1-d Q_2)]'}
  =0.
  \]
  Therefore it holds that
  \begin{align}
    \sum_{\tr \rho(\mu^p \lambda^q)= c} \frac{1}{\BT_{M_n, \gamma}(\rho)}
    &= -\left. \frac{R}{[y(Q_1-d Q_2)]'} \right|_{y=0} \notag\\
    &= -\left. \frac{R}{(Q_1-d Q_2)} \right|_{y=0} \notag\\
    &= -\left. \frac{R/Q_2}{(Q_1-d Q_2)/Q_2} \right|_{y=0} \notag\\
    &= -\left.\frac{-\frac{k_n(y)}{\hat \ell_n(y)}2(g^2-h^2)}{Q-c^2}\right|_{y=0}. \label{eq:sumToddpn2}
  \end{align}
  The value of $g^2 - h^2$ at $y=0$ turns out to be
  \[
  g^2(0)-h^2(0)
  = j_{p-4q}(s(0)) \ell_{p-4q}(0)
  = f_{p-4q}(s(0))
  \]
  since 
  \begin{align*}
    g(0) &= \delta_{(p+1)/2-2q}(0) =f_{(p+1)/2-2q}(s(0)),\\
    h(0) &= \delta_{(p-1)/2-2q-1}(0) = f_{(p-1)/2-2q}(s(0))
  \end{align*}
  by $u=0$ and $v=1$ at $y=0$.
  By $Q=(m^{p-4q}+m^{-(p-4q)})^2 = m^{2(p-4q)} + m^{-2(p-4q)} +2 = T_{p-4q}(s(0)) + 2$ and
  $k_n(0)/\ell_n(0) = s(0)=-4/n$,
  we can rewrite Eq.~\eqref{eq:sumToddpn2} as
  \[
  \sum_{\tr \rho(\mu^p \lambda^q)= c} \frac{1}{\BT_{M_n, \gamma}(\rho)}
  = -\frac{4}{n} \frac{2f_{p-4q}(-4/n)}{T_{p-4q}(-4/n)+2-c^2}.
  \]
\end{proof}

We can also carry out the sum of $(\BT_{M_n, \gamma}(\rho))^{-1}$ on $L$.
\begin{proposition}
  \label{prop:sumTonLoddp}
  On the component $L$, we have
  \[
  \sum_{\tr \rho(\mu^p \lambda^q)= c} \frac{1}{\BT_{M_n, \gamma}(\rho)}
  = \frac{8}{n} \frac{f_{p-4q}(-4/n)}{T_{p-4q}(-4/n)+2-c^2}.
  \]
\end{proposition}
\begin{proof}
  Recall that $z=m+m^{-1}$. We use the symbol $t$ for $m^2 + m^{-2} = z^2 -2$.
  The trace $\tr \rho(\mu^p\lambda^q)$ turns into $m^{p-4q} + m^{-(p-4q)} = T_{p-4q}(z)$ by $\ell=m^{-4}$.
  We have
  \begin{align*}
    \sum_{\tr \rho(\mu^p \lambda^q)= c} \frac{1}{\BT_{M_n, \gamma}(\rho)}
    &= \sum_{T_{p-4q}(z) = c} \frac{1}{\BT_{M_n, \gamma}(\rho)} \\
    &= \frac{1}{2} \sum_{T_{p-4q}^2(z) = c^2} \frac{1}{\BT_{M_n, \gamma}(\rho)} \\
    &= \frac{1}{2} \sum_{T_{2p-8q}^2(z) = c^2-2} \frac{1}{\BT_{M_n, \gamma}(\rho)} \\
    &= \sum_{T_{p-4q}(t) = c^2-2} \frac{1}{\BT_{M_n, \gamma}(\rho)} \\
    &= \sum_{T_{p-4q}(t) = c^2-2} \frac{-4}{p-4q} \frac{1}{2+t\frac{n}{2}} \\
    &= -\frac{8}{n} \sum_{T_{p-4q}(t) = c^2-2} \frac{1}{p-4q} \frac{1}{t+\frac{4}{n}} \\
    &= \frac{8}{n} \frac{f_{p-4q}(-4/n)}{T_{p-4q}(-4/n)-(c^2-2)}
  \end{align*}
  by Lemma~\ref{lemma:sum_t-b}.
\end{proof}
Combining Propositions~\ref{prop:sumTonDoddp} and~\ref{prop:sumTonLoddp},
we obtain Theorem~\ref{thm:main} in the case that $n \equiv 2 \pmod{4}$ and $p$ is odd.  

\section{Torus knot exteriors}
Let $r$, $s$ be coprime integers and $T_{r, s}$ the torus knot of type $(r, s)$. 
Set $M_{r, s}= S^3 \setminus \mathrm{Int} N(T_{r, s})$ where $N(T_{r, s})$ is a neighborhood of $T_{r, s}$.
We show that the sum of $\BT_{M_{r, s}, \gamma}^{-1}(\rho)$ equals $\pm 2$ for all slopes $\gamma$.
For simplicity, we assume $r > 0$ and $s > 0$.

\subsection{The sum for $\gamma=\mu$}
When we present $\pi_1(M_{r, s})$ as 
\begin{equation}
  \label{eq:pres_torusknotgroup}
  \pi_1(M_{r, s}) = \la x, y \,|\, x^r = y^s \ra,
\end{equation} 
according to~\cite{Johnson}, 
the character variety $X(\pi_1(M_{r, s}))$ consists of $(r-1)(s-1)/2$ components which are indexed by the following integers $a$ and $b$:
\begin{itemize}
\item $0 < a < r$, $0 < b < s$;
\item $a \equiv b \pmod{2}$ and;
\item $\tr \rho(x) = 2 \cos (\pi a / r)$, $\tr \rho(y) = 2 \cos (\pi b / s)$ for any conjugacy class $[\rho]$ on its component.
\end{itemize}
Each component in $X(\pi_1(M_{r, s}))$ is a complex line $\BC$ with local parameter $\tr_{\mu}$ of the meridian $\mu$.
Under the presentation~\eqref{eq:pres_torusknotgroup}
the longitude $\lambda$ satisfies
\[\lambda = x^r \mu^{-rs}.\]
The adjoint Reidemeister torsion $\BT_{M_{r, s}, \mu}(\rho)$ turns to be locally constant on $X(\pi_1(M_{r, s}))$ as follows.
\begin{lemma}
  \label{lemma:torsion_torus_knot_meridian}
  For an $\SL_2(\BC)$-representation $\rho$ such that $[\rho]$ lies in the component with index $(a, b)$, the inverse of $\BT_{M_{r, s}, \mu}(\rho)$ is expressed as
  \[(\BT_{M_{r, s}, \mu}(\rho))^{-1}
  = \frac{16}{rs} \sin^2 \left(\frac{\pi a}{r}\right) \sin^2 \left( \frac{\pi b}{s} \right).\]
\end{lemma}
\begin{proof}
  The eigenvalues $m$ and $\ell$ for $\rho(\mu)$ and $\rho(\lambda)$ satisfy that 
$\ell = (-1)^a m^{-rs}$.
  By the curve change formula~\eqref{eq:curve_change}, we have
  \begin{align*}
    (\BT_{M_{r, s}, \mu}(\rho))^{-1}
    &= \left( \frac{d \log m}{d \log \ell} \right)^{-1} (\BT_{M_{r, s}, \lambda}(\rho))^{-1} \\
    &= \left( \frac{1}{-rs} \right)^{-1} (\BT_{M_{r, s}, \lambda}(\rho))^{-1}.
  \end{align*}
  The lemma follows from the result in~\cite[\S~6.2]{Dubois} that
  \[(\BT_{M_{r, s}, \lambda}(\rho))^{-1}
  = - \frac{16}{r^2s^2} \sin^2 \left(\frac{\pi a}{r}\right) \sin^2 \left( \frac{\pi b}{s} \right).\]
  Note that the torsion in~\cite{Dubois} turns into the inverse of our torsion in this paper according to the difference in the conventions.
\end{proof}
We can now carry out the inverse sum of adjoint Reidemeister torsions for the torus knot exterior $M_{r, s}$ and the meridian $\mu$.
\begin{theorem}
  \label{thm:sum_meridian}
  For a generic complex number $c$, it holds that 
  \begin{equation}
    \label{eq:sum_meridian}
    \sum_{[\rho] \in \tr_{\mu}^{-1}(c)} \frac{1}{\BT_{M_{r, s}, \mu}(\rho)} = 2.
  \end{equation}
\end{theorem}
\begin{proof}
  Since we can regard every component in $X(\pi_1(M_{r, s}))$ as a complex line $\BC$ by the function $\tr_{\mu}$, there exists a single conjugacy class $[\rho]$ such that $\tr \rho(\mu) = c$ on each component. The inverse of $\BT_{M_{r, s}, \mu}(\rho)$ runs over all values once in the summation.

  In the case of even $r$ and odd $s$,
  we can express the sum of $(\BT_{M_{r, s}, \mu}(\rho))^{-1}$ as 
  \begin{align}
    &\sum_{[\rho] \in \tr_{\mu}^{-1}(c)} \frac{1}{\BT_{M_{r, s}, \mu}(\rho)} \notag \\
    &= \frac{16}{rs}\left\{
    \left(
    \sin^2\left(\frac{\pi}{r}\right) + \cdots + \sin^2\left(\frac{(r-1)\pi}{r}\right)
    \right)
    \left(
    \sin^2\left(\frac{\pi}{s}\right) + \cdots + \sin^2\left(\frac{(s-2)\pi}{s}\right)
    \right) \right. \notag \\
    &\quad \left.
    + \left(
    \sin^2\left(\frac{2\pi}{r}\right) + \cdots + \sin^2\left(\frac{(r-2)\pi}{r}\right)
    \right)
    \left(
    \sin^2\left(\frac{2\pi}{s}\right) + \cdots + \sin^2\left(\frac{(s-1)\pi}{s}\right)
    \right)
    \right\}. \label{eq:sum_torsion_torusknot}
  \end{align}
  from Lemma~\ref{lemma:torsion_torus_knot_meridian}.
  It follows from
    $\sin^2(\pi/s) + \cdots + \sin^2((s-2)\pi/s)
    =\sin^2(\pi - \pi/s) + \cdots + \sin^2(\pi - (s-2)\pi/s)
    =\sin^2((s-1)\pi/s) + \cdots + \sin^2(2\pi/s)$
  that 
  \begin{align*}
    \sin^2\left(\frac{\pi}{s}\right) + \cdots + \sin^2\left(\frac{(s-2)\pi}{s}\right)
    &= \frac{1}{2} \sum_{k=1}^{s-1} \sin^2 \left(\frac{k\pi}{s}\right) \\
    &= \frac{1}{2} \sum_{k=1}^{s-1} \left( \frac{1}{2\sqrt{-1}} (e^{\frac{k\pi\sqrt{-1}}{s}} - e^{-\frac{k\pi\sqrt{-1}}{s}}) \right)^2 \\
    &= \frac{1}{2} \left( \frac{s-1}{2} - \frac{1}{4} \sum_{k=1}^{s-1} (e^{\frac{2k\pi\sqrt{-1}}{s}} + e^{-\frac{2k\pi\sqrt{-1}}{s}})\right) \\
    &= \frac{s}{4}
  \end{align*}
  Similarly we have $\sum_{k=1}^{r-1} \sin^2( k\pi / r) = r/2$.
  The sum~\eqref{eq:sum_torsion_torusknot} turns to be
  \[
    \eqref{eq:sum_torsion_torusknot}
    = \frac{16}{rs} \left(\sum_{k=1}^{r-1} \sin^2\frac{k\pi}{r}\right) \frac{s}{4} 
    = 2.
  \]
  In the case of odd $r$ and odd $s$, we can also have
  \[\sum_{[\rho] \in \tr_{\mu}^{-1}(c)} \frac{1}{\BT_{M_{r, s}, \mu}(\rho)}
  = \frac{16}{rs} \left(\frac{r}{4} \frac{s}{4} + \frac{r}{4} \frac{s}{4}\right)
  =2.\]
\end{proof}
\begin{remark}
  In the case of $r=2$, we can find another approach to compute the inverse sum of $\BT_{M_{2, s}, \mu}(\rho)$ in~\cite[Remark~1.5]{Yoon2}.
\end{remark}

\subsection{The sum for a general slope}
Finally we are going to carry out the sum of $(\BT_{M_{r, s}, \gamma}(\rho))^{-1}$ for all slopes $\gamma = \mu^p\lambda^q$.
\begin{theorem}
  \label{thm:sum_torusknot_general}
  For any slope $\gamma = \mu^p\lambda^q$ and a generic $c \in \BC$, it holds that
  \[\sum_{[\rho] \in \tr_{\gamma}^{-1}(c)} \frac{1}{\BT_{M_{r, s}, \gamma}(\rho)}
  = \frac{|p-rsq|}{p-rsq} 2.\]
\end{theorem}
\begin{proof}
  The slope $\gamma$ satisfies that $\gamma = \mu^p \lambda^q = x^{rq} \mu^{p-rsq}$.
  The eigenvalues $\xi^{\pm 1}$ of $\rho(\gamma)$ satisfies
  $\xi = (-1)^{aq} m^{p-rsq}$ on the component with index $(a, b)$ in $X(\pi_1(M_{r, s}))$.
  By the curve change formula~\eqref{eq:curve_change},
  we can express $(\BT_{M_{r, s}, \gamma}(\rho))^{-1}$ as
  \[(\BT_{M_{r, s}, \gamma}(\rho))^{-1}
  = \left( \frac{d \log \xi}{d \log m}\right)^{-1} (\BT_{M_{m, n}, \mu}(\rho))^{-1}
  = \frac{1}{p-rsq} (\BT_{M_{r, s}, \mu}(\rho))^{-1}.\]
  Recall from Proposition~\ref{prop:sumTonDevenp} that $T_n(z)= f_{n+1}(z) - f_{n-1}(z)$.
  If we put $z=m+m^{-1}$, then $T_n$ satisfies
  $T_n(m+m^{-1}) = m^n + m^{-n}$.
  We can rewrite the condition $\tr \rho (\gamma) = c$ on the component with index $(a, b)$ as
  \[ T_{p-rsq}(z) = (-1)^{aq}c \]
  for $z = \tr \rho(\mu) = m+m^{-1}$.
  Actually $T_n$ is the Chebyshev polynomial of the first kind with degree $|n|$.
  It is known that $T_n$ has $|n|$ different simple roots.
  Since $z=\tr\rho(\mu)$ is a local coordinate on every component in $X(\pi_1(M_{r, s}))$, there exists $|p - rsq|$ conjugacy classes $[\rho]$ such that $\tr \rho(\gamma)=c$ on each component.
  By Theorem~\ref{thm:sum_meridian}, we obtain
  \begin{align*}
    \sum_{[\rho] \in \tr_{\gamma}^{-1}(c)} \frac{1}{\BT_{M_{r, s}, \gamma}(\rho)}
    &= \frac{1}{p-rsq} \sum_{[\rho] \in \tr_{\gamma}^{-1}(c)} \frac{1}{\BT_{M_{r, s}, \mu}(\rho)} \\
    &= \frac{|p-rsq|}{p-rsq}2.
  \end{align*}
\end{proof}
\begin{corollary}
  \label{cor:nonvanish_torusknot}
  The vanishing identity of the adjoint Reidemeister torsion does not hold for any torus knot exteriors and any slopes.
\end{corollary}
\section*{Acknowledgements}
The authors wish to express their thanks to D.~Gang, S.~Kim, and S.~Yoon for helpful conversations and pointing out a mistake in the previous version.  
The first author has been supported by grants from the Simons Foundation (\#354595 and \#708778).
The second author has been supported by JSPS KAKENHI Grant Number 21K03242.

\end{document}